\documentclass[12pt,english]{amsart}
\usepackage{dsfont,amssymb,amsmath,mathrsfs,fourier,baskervald}

\usepackage[text={6.5in,8.5in},centering]{geometry}

\usepackage{enumitem}
\usepackage{booktabs,longtable,array}
\usepackage{float}
\usepackage[dvipsnames]{xcolor}   
\usepackage{xparse}
\usepackage{xy}
\xyoption{all}
\usepackage{xr-hyper}
\usepackage[linktocpage=true,colorlinks=true,hyperindex,citecolor=teal,linkcolor=blue]{hyperref}

\newtheorem{theorem}{Theorem}[section]
\newtheorem{lemma}[theorem]{Lemma}
\newtheorem{corollary}[theorem]{Corollary}
\theoremstyle{definition}
\newtheorem{definition}{Definition}[section]
\newtheorem{example}[theorem]{Example}
\numberwithin{equation}{section}
\theoremstyle{remark}
\newtheorem{remark}[theorem]{Remark}

\newcommand{\ind}[1]{\mathcal{I}_\top(#1)}
\newcommand{\indn}[2]{\mathcal{I}^{#2}_\top(#1)}
\newcommand{\ded}[1]{\mathcal{D}_\top(#1)}
\newcommand{\dedn}[2]{\mathcal{D}^{#2}_\top(#1)}
\newcommand{\clo}[1]{\mathcal{C}_\top(#1)}

\renewcommand{\top}{\ast}

\begin{document}
	
\title[From subtractive ideals of semirings to deductive and inductive sets in general algebras]{From subtractive ideals of semirings to deductive and inductive sets in general algebras}
	
\author{Elena Caviglia}
	
\address{National Institute for Theoretical and Computational Sciences, South Africa}
\email{elena.caviglia@outlook.com}

\author{Amartya Goswami}
\address{[1] Department of Mathematics and Applied Mathematics, University of Johannesburg, South Africa; [2] National Institute for Theoretical and Computational Sciences}
\email{agoswami@uj.ac.za}

\author{Zurab Janelidze }
\address{[1] Mathematics Division, Department of Mathematical Sciences, Stellenbosch University, Private Bag
X1, Matieland, 7602 Stellenbosch, Western Cape, South Africa; [2] National Institute for Theoretical and Computational Sciences}
\email{zurab@sun.ac.za}

\author{ Luca Mesiti}
\address{Mathematics Division, Department of Mathematical Sciences, Stellenbosch University, Private Bag
X1, Matieland, 7602 Stellenbosch, Western Cape, South Africa}
\email{luca.mesiti@outlook.com}
	
\author{Vaino T. Shaumbwa}
\address{Department of Computing, Mathematical, and Statistical Sciences,
University of Namibia,
Private Bag 13301,
Windhoek,
Namibia}
	\email{vshaumbwa@unam.na}
	
    
\thanks{The third author used Chat GPT-6 Astra to assist with research, writing, and the identification of relevant literature. Responsibility for the correctness of the final manuscript and the accuracy of the references rests with the authors.}
	
	\begin{abstract}
Kernels of semiring homomorphisms are precisely the subtractive ideals. We extend this decomposition of normality to general algebras by introducing inductive and deductive sets, which turn out to correspond to the two directions of the biconditional in Mal'tsev's criterion for a congruence class. In semirings, inductivity recovers idealhood for subsets containing zero, and deductivity of an ideal recovers subtractivity. Their generation processes, described by polynomials or equivalently by reflexive compatible relations (semicongruences), give two ranks measuring the numbers of steps required uniformly in a variety. We show that apart from the trivial cases, any pair of positive integers or infinity is the inductive-deductive rank pair of some variety. For the variety of semirings, the rank pair is $(1,\infty)$, while for any non-trivial Mal'tsev variety, it is $(1,1)$. For varieties of finite-group actions, inductive rank is expressed exactly in terms of directed Cayley-graph diameters, while deductive rank is related to undirected diameters after symmetrization. We compute both ranks of these action varieties for every finite abelian group in terms of its invariant factors, and obtain their complete spectrum. We also establish special spectrum theorems for subtractive varieties and several classes of ordered algebras. Finally, we characterise structural properties of algebras and varieties by conditions on induction and deduction.
\end{abstract}
	
	\maketitle 
	
	\section{Introduction}

Kernels of semiring homomorphisms are characterised as subtractive ideals, also called $k$-ideals,  introduced by Henriksen \cite{Hen58} (see also \cite{AA08, Gol99, Han15, JRT22, LaG95, LaT65, Les15, SA92, SA93, WSA96, DG24}).  An ideal $I$ of a semiring is a kernel if and only if
\[
x+y\in I,\quad y\in I\quad\Longrightarrow\quad x\in I.
\]
A congruence class containing $0$ is called \emph{normal}. We can therefore express this characterisation schematically as
\begin{equation}\label{EquG}
\text{ideal}+\text{subtractive}=\text{normal}.
\end{equation}
Does this decomposition have a counterpart for general algebras? Our work started with identifying suitable counterparts of idealhood and subtractivity for general algebras. We called them \emph{$0$-inductivity} and \emph{$0$-deductivity}, for reasons to be explained below. For a subset $I$ of an algebra with a distinguished constant $0$, they can be expressed using polynomial maps $q\colon A^m\to A$ (term operations with parameters from the algebra):
\begin{align*}
\text{$0$-inductivity:}\qquad
&\vec a\in I^m,\quad q(\vec0)\in I\quad\Longrightarrow\quad q(\vec a)\in I;\\
\text{$0$-deductivity:}\qquad
&\vec a\in I^m,\quad q(\vec a)\in I\quad\Longrightarrow\quad q(\vec0)\in I.
\end{align*}
Here $m$ ranges over the positive integers and $\vec0=(0,\ldots,0)$. For non-empty subsets, the resulting decomposition is
\begin{equation}\label{IntroSplit}
0\text{-inductive}+0\text{-deductive}=0\text{-normal}.
\end{equation}

As H.~Peter Gumm pointed out to us (private communication), this decomposition follows from Mal'tsev's original criterion for congruence classes \cite{Mal54}. Agliano and Ursini \cite[Proposition 1.4(4)]{AU92} give the following pointed unary formulation: a subset $I$ is $0$-normal if and only if $0\in I$ and, for every $a\in I$ and every unary polynomial $g$,
\[
g(a)\in I\quad\Longleftrightarrow\quad g(0)\in I.
\]
For $0\in I$, the implication from right to left is $0$-inductivity; its converse is $0$-deductivity. Replacing coordinates one at a time gives the multivariate forms. Moreover, a non-empty $0$-deductive set contains $0$, by taking the identity polynomial. Thus \eqref{IntroSplit} splits the classical criterion into its two implications. The purpose of this paper is to study these separately, together with their associated generation processes.

In semirings, $0$-inductive sets containing $0$ are precisely the ideals, and an ideal is $0$-deductive if and only if it is subtractive. For an arbitrary subset $J$ of a semiring, deductivity takes the following form: $J$ is $0$-deductive if and only if $x+y\in J$ and $y\in\overline J$ always imply $x\in J$, where $\overline J$ is the ideal generated by $J$ (Theorem~\ref{ThmD}).

Ideals in semirings are also the universal-algebraic ideals in the usual ideal-term sense. One may therefore ask why we do not instead abstract subtractivity by a property which, together with universal-algebraic idealhood, characterises normality. This question is meaningful since in general varieties, for subsets containing $0$, inductivity is stronger than idealhood. In fact, an ideal can be deductive without being normal, as Example~\ref{ExIdealDeductive} shows. While there may be other ways of abstracting equation~\eqref{EquG} to general algebras, the solution proposed here has a natural dual relational interpretation. This interpretation, discussed below, and the connection with Mal'tsev's criterion provide the motivation for our choice.

Let us first note that the two algebraic conditions defining inductive and deductive sets immediately suggest procedures for generating such sets. For an arbitrary subset $I\subseteq A$, define
\begin{align*}
\mathcal I_0(I)
&=\{q(\vec a)\mid \vec a\in I^m,\ q(\vec0)\in I\},\\
\mathcal D_0(I)
&=\{q(\vec0)\mid \vec a\in I^m,\ q(\vec a)\in I\},
\end{align*}
where $q$ ranges over polynomial maps of all positive finite arities $m$, with arbitrary parameters from $A$. These same operators have an equivalent relational description. Let $R_I$ be the least reflexive compatible binary relation on $A$ containing $I\times\{0\}$; compatibility means that the relation is preserved by every basic operation. Then
\begin{align*}
\mathcal I_0(I)&=R_I I
 =\{x\in A\mid xR_I y\text{ for some }y\in I\},\\
\mathcal D_0(I)&=I R_I
 =\{x\in A\mid yR_I x\text{ for some }y\in I\}.
\end{align*}
Indeed, applying terms to the generating pairs $(a,0)$ and diagonal pairs supplying the parameters gives the polynomial pairs $(q(\vec a),q(\vec0))$; together with the diagonal, these constitute $R_I$. Thus the polynomial and relational formulae describe the same single step. Each operator contains its input, and its fixed sets are precisely the corresponding inductive or deductive sets.

This relational interpretation is also what gives the motivation for the choice of terminology ``inductive/deductive''. Regard $xRy$, where $R=R_I$, as a generalised implication $x\to y$, and membership in $I$ as being accepted as true. Inductivity then reads
\[
x\to y,\quad y\in I\quad\Longrightarrow\quad x\in I:
\]
from ``$x$ implies $y$'' and ``$y$ is true'', one passes to ``$x$ is true'', as in \emph{inductive reasoning} from a consequence to a possible premise. Deductivity reads
\[
x\to y,\quad x\in I\quad\Longrightarrow\quad y\in I:
\]
from ``$x$ implies $y$'' and ``$x$ is true'', one passes to ``$y$ is true'', as in \emph{deductive reasoning}. This analogy has a familiar instance in logic: in Heyting algebras with $1$ distinguished, our non-empty deductive sets are precisely the deductive filters of intuitionistic propositional logic (Theorem~\ref{PropHeyting}).

Starting with $I$, we can repeatedly apply the chosen operator to the set obtained at the previous step:
\[
I\subseteq\mathcal I_0(I)\subseteq\mathcal I_0^2(I)\subseteq\cdots,
\qquad
I\subseteq\mathcal D_0(I)\subseteq\mathcal D_0^2(I)\subseteq\cdots.
\]
At each step, the relation is generated afresh from the current set $J$, giving $R_JJ$ or $JR_J$. Finitarity of the polynomials implies that the unions are the least $0$-inductive and $0$-deductive sets containing $I$. The number of steps leads to two invariants. The \emph{$0$-inductive rank} of a variety containing a constant $0$ is the least non-negative integer $r$ such that $\mathcal I_0^r(I)=\mathcal I_0^{r+1}(I)$ for every non-empty subset $I$ of every algebra in the variety; the \emph{$0$-deductive rank} is defined in the same way. Rank $\infty$ means that there is no finite uniform bound. Our \emph{general rank spectrum theorem} (Theorem~\ref{ThmCompleteSpectrum}) states that the pairs of inductive and deductive ranks of varieties with a distinguished constant are exactly
\begin{equation}\label{IntroSpectrum}
\{(0,0),(0,1)\}\ \cup\
\bigl(\overline{\mathbb N}_{>0}\bigr)^2,
\end{equation}
where $\overline{\mathbb N}_{>0}=\{1,2,3,\ldots\}\cup\{\infty\}$.
Moreover, every pair of positive finite ranks is realized by a variety generated by a finite algebra in a finite signature. Thus, apart from the restrictions involving rank zero, the two generation processes have independent complexities. Neither finite rank bounds the other, even when one of them is one. Section~\ref{SecRanks} proves this result by constructing three families of varieties briefly discussed below.

Congruence $(r+1)$-permutability bounds both ranks by $r$, and finite algebras satisfying the Hagemann--Mitschke identities attain the pair $(r,r)$. Semilattices with mutually annihilating meet-endomorphisms give every pair $(a,1)$ with $a$ positive and finite. Bounded distributive lattices with endomorphisms and dual endomorphisms give every pair $(1,b)$ with $b$ positive and finite. In the latter proof, finite lattice duality represents the algebras as lattices of upward-closed subsets of finite ordered sets. The points of these ordered sets have integer heights. The operations make deduction at a given height depend on membership at lower heights, and suitable order relations force this dependence to require successively more steps.

An independent product, whose category of algebras is equivalent to the product of the factor categories, combines the required ranks. With one factor from the Hagemann--Mitschke family, our product theorem gives the componentwise maxima; this need not hold for arbitrary independent products. Meet-semilattices with a greatest element and a contractive meet-endomorphism supply $(\infty,1)$, while commutative monoids and semirings supply $(1,\infty)$. Products give all the remaining pairs with an infinite coordinate. The trivial variety and sets with a distinguished constant supply $(0,0)$ and $(0,1)$, respectively, and the proof excludes every other pair involving zero.

The general rank spectrum theorem leaves open an equally natural question: which pairs occur when the variety is required to belong to a specified class? We call a complete answer to such a question a \emph{special rank spectrum theorem}. Section~\ref{SecSpecialSpectra} establishes such theorems for subtractive varieties, monotone semilattice expansions, Cornish algebras, and varieties of finite abelian group actions. It also determines the spectra of all non-trivial subvarieties of each of the two families $\mathcal Q_N$ and $\mathcal C_m$ used in the general theorem. Establishing further special rank spectrum theorems for classes defined by other identities or Mal'tsev conditions would be of interest.

Group actions provide a connection with generation lengths in finite groups. For a finite group $G$, let $\mathcal A_G$ be the variety of $G$-sets with an arbitrary named element $c$, not required to be fixed by the action. Its inductive rank is the rounded-up base-two logarithm of one plus the largest directed Cayley-graph diameter over generating sets of subgroups of $G$ (Theorem~\ref{ThmGroupInductiveDiameter}). The $k$th induction stage is described by words of length at most $2^k-1$. Deduction instead produces undirected word balls after an initial symmetrization, giving the bounds and exact description in Theorem~\ref{ThmGroupDeductiveDiameter}. Using the diameter theorems of Klopsch and Lev \cite{KL03,KL09}, we compute both ranks for every finite abelian group (Theorem~\ref{ThmAbelianActionRanks}). Induction is governed by the positive diameter; deduction is governed jointly by the symmetric diameter and the exponent. The resulting special spectrum consists of $(0,1)$, all finite positive diagonal pairs, and the neighbouring pairs $(r,r+1)$ and $(r+1,r)$ for $r\geqslant2$. For example, $C_4$-actions give $(2,3)$, whereas $C_3^2$-actions give $(3,2)$. The cyclic-action spectrum follows as a special case.

One can also require a rank bound to hold at every chosen element of every algebra, rather than at a fixed named constant. These \emph{universal ranks} are exactly the ranks of the variety obtained by adjoining an unconstrained constant (Theorem~\ref{ThmUniversalConstantReduction}). Thus the group-action results also compute universal ranks for varieties of $G$-sets. The universal spectrum over all varieties remains a special rank spectrum whose determination is an interesting question.

In the last section, Section~\ref{SecDetection}, we show that induction
and deduction can detect structural properties of algebras and their
varieties. These include subtractivity and distributivity, and a number of other structural properties.

\clearpage
\begingroup
\small
\tableofcontents
\endgroup

\section{Induction and deduction of sets}\label{SecClosures}

Throughout this section, $A$ is a finitary algebra with a chosen element $\top$. A \emph{semicongruence} on $A$ is a reflexive compatible binary relation, that is, a reflexive relation which is a subalgebra of $A\times A$. For a relation $S$ and a subset $I$, we write
\[
SI=\{x\in A\mid \exists y\in I\,(xSy)\},
\qquad
IS=\{x\in A\mid \exists y\in I\,(ySx)\}.
\]
We denote the opposite relation by $S^\circ$, and use $ST$ for the composite given by $x(ST)z$ iff $xSyTz$ for some $y$. Let $R_I$ be the semicongruence generated by $I\times\{\top\}$; when $I$ is fixed, we write simply $R$.

\begin{definition}
For a subset $I$ of $A$, its \emph{$\top$-induction} is
\[
\ind I=R_I I.
\]
We say that $I$ is \emph{$\top$-inductive} when $\ind I\subseteq I$.
\end{definition}

By reflexivity, $I\subseteq\ind I$, so $I$ is $\top$-inductive if and only if $I=\ind I$.

\begin{definition}
For a subset $I$ of $A$, its \emph{$\top$-deduction} is
\[
\ded I=I R_I.
\]
We say that $I$ is \emph{$\top$-deductive} when $\ded I\subseteq I$.
\end{definition}

Again, $I\subseteq\ded I$, so $I$ is $\top$-deductive if and only if $I=\ded I$. Equivalently, inductivity requires $SI\subseteq I$, and deductivity requires $IS\subseteq I$, for some semicongruence $S$ containing $I\times\{\top\}$: the least such semicongruence is $R_I$. The empty set has both properties. Every non-empty deductive set contains $\top$, since $aR_I\top$ for $a\in I$.

\subsection{Polynomial descriptions and normality}

The relation $R_I$ consists of all pairs
\[
\bigl(t(\vec x,\vec y),t(\vec x,\vec\top)\bigr),
\]
where $t$ is a term, $\vec x$ is a finite tuple of parameters from $A$, and $\vec y$ is a finite tuple from $I$. Tuples of length zero are allowed; in particular, terms involving only parameters supply the diagonal. We obtain the following descriptions.

\begin{lemma}\label{LemA}
$\ind I=\{t(\vec x,\vec y)\mid \vec y\in I,\ t(\vec x,\vec\top)\in I\}$.
\end{lemma}

\begin{lemma}\label{LemB}
$\ded I=\{t(\vec x,\vec\top)\mid \vec y\in I,\ t(\vec x,\vec y)\in I\}$.
\end{lemma}

Here and below, $\vec y\in I$ means that every entry of the tuple belongs to $I$. Thus the multivariate polynomial operators in the Introduction are exactly the relational operators $R_I I$ and $I R_I$, with $\top=0$; these are two descriptions of the same one-step constructions. Their fixed sets admit the following unary characterisations.

\begin{samepage}
\begin{lemma}\label{PropUnary}
For a subset $I$ of $A$:
\begin{enumerate}
\item $I$ is $\top$-inductive if and only if, for every unary polynomial $g$ and every $a\in I$,
\[
g(\top)\in I\quad\Longrightarrow\quad g(a)\in I.
\]
\item $I$ is $\top$-deductive if and only if, for every unary polynomial $g$ and every $a\in I$,
\[
g(a)\in I\quad\Longrightarrow\quad g(\top)\in I.
\]
\end{enumerate}
\end{lemma}
\end{samepage}

\begin{proof}
The forward implications follow by taking one ideal variable in Lemmas~\ref{LemA} and \ref{LemB}. Conversely, for induction, replace the entries of $\vec\top$ in $t(\vec x,\vec\top)$ by the corresponding entries of $\vec y$ one at a time. At each step the remaining entries are parameters, so the unary condition applies. For deduction, start with $t(\vec x,\vec y)$ and replace the entries of $\vec y$ by $\top$ one at a time.
\end{proof}

\begin{remark}\label{RemUnarySteps}
The unary one-step operators
\begin{align*}
\mathcal I_\top^{\mathrm u}(I)
&=\{g(a)\mid a\in I,\ g(\top)\in I,\ g\text{ a unary polynomial}\},\\
\mathcal D_\top^{\mathrm u}(I)
&=\{g(\top)\mid a\in I,\ g(a)\in I,\ g\text{ a unary polynomial}\}
\end{align*}
are contained in $\ind I$ and $\ded I$, respectively, and have the same fixed sets by Lemma~\ref{PropUnary}. Since these operators are extensive, monotone and finitary, their increasing iterations have the same unions as the corresponding multivariate iterations. The one-step inclusions can nevertheless be strict.

For example, in the additive monoid $\mathbb N_0^3$, every unary polynomial has the form $g(z)=b+nz$. Let $e_1,e_2,e_3$ be the standard basis vectors. For $I=\{0,e_1,e_2,e_3\}$, the vector $e_1+e_2+e_3$ belongs to $\mathcal I_0(I)$, by using $q(z_1,z_2,z_3)=z_1+z_2+z_3$. It does not belong to $\mathcal I_0^{\mathrm u}(I)$: a vector $b+na$ with $a,b\in I$ has at most two nonzero coordinates.

For deduction, take $J=\{0,e_1,e_2,e_1+e_2+e_3\}$. The polynomial $q(z_1,z_2)=e_3+z_1+z_2$ shows that $e_3\in\mathcal D_0(J)$. But $e_3+na\notin J$ for every $a\in J$ and $n\in\mathbb N_0$, so $e_3\notin\mathcal D_0^{\mathrm u}(J)$.

In particular, a unary replacement of the operators would change the rank invariant. In $\mathbb N_0^d$, starting from $\{0,e_1,\ldots,e_d\}$, after $n$ unary induction steps every vector has at most $2^n$ nonzero coordinates. Thus unary induction has no finite uniform rank in commutative monoids, whereas the multivariate induction of this paper has rank $1$. Deductive rank is infinite for both versions. Indeed, Theorem~\ref{ThmMonoids} proves this for multivariate deduction. If unary deduction reached its closure uniformly in $n$ steps, the inclusions of its iterates in the multivariate iterates, together with equality of the closures, would give the same bound for multivariate deduction. The unary rank pair is therefore $(\infty,\infty)$, whereas the multivariate rank pair is $(1,\infty)$.
\end{remark}

A subset is \emph{$\top$-normal} if it is the congruence class of $\top$ for some congruence on $A$.

We recall explicitly the classical criterion underlying the next theorem. Mal'tsev's criterion \cite{Mal54} characterises a non-empty congruence class $C\subseteq A$ by the requirement
\[
\tau(C)\cap C\neq\varnothing
\quad\Longrightarrow\quad
\tau(C)\subseteq C
\]
for every unary polynomial $\tau$; compare \cite[Proposition 1.4(3)]{AU92} for the pointed version. An equivalent multivariate reformulation is
\[
q(C^m)\cap C\neq\varnothing
\quad\Longrightarrow\quad
q(C^m)\subseteq C
\]
for every polynomial $q\colon A^m\to A$, $m\geqslant1$. Indeed, starting with any tuple in $C^m$ whose image belongs to $C$, the unary criterion allows its coordinates to be replaced successively by those of any other tuple in $C^m$, while keeping the image in $C$. The converse is the case $m=1$.

If $\top\in C$, the multivariate condition is equivalently
\[
q(\vec a)\in C\quad\Longleftrightarrow\quad q(\vec\top)\in C
\qquad(\vec a\in C^m).
\]
The unary case is the biconditional stated in \cite[Proposition 1.4(4)]{AU92}. By Lemmas~\ref{LemA} and \ref{LemB}, its multivariate right-to-left and left-to-right implications are respectively inductivity and deductivity. The next theorem records this splitting of Mal'tsev's criterion.

\begin{theorem}\label{ThmA}
A non-empty subset $I$ of $A$ is $\top$-normal if and only if it is both $\top$-inductive and $\top$-deductive.
\end{theorem}

\begin{proof}
By Lemma~\ref{PropUnary}, the conjunction is precisely the biconditional in \cite[Proposition 1.4(4)]{AU92}, with $\top$ in place of $0$. A non-empty deductive set contains $\top$, so Mal'tsev's criterion applies. Naming $\top$ by an additional constant, if necessary, changes neither the congruences nor the polynomial maps of $A$.
\end{proof}

For completeness, the same conclusion follows directly from the relation construction. If $C$ is the congruence generated by $I\times\{\top\}$, then
\begin{equation}\label{EquF}
C=R\cup R^\circ R\cup RR^\circ R\cup R^\circ RR^\circ R\cup\cdots.
\end{equation}
Reflexivity makes this an increasing union. It is the equivalence relation generated by $R$, and is compatible. A normal set is closed under both directions of $R$. Conversely, if $I$ is non-empty and closed under both directions, then $\top\in I$, and following any finite $R$-zigzag from $\top$ shows that $[\top]_C\subseteq I$. The reverse inclusion follows from $I\times\{\top\}\subseteq C$.

We thank H.~Peter Gumm for directing us to Mal'tsev's original formulation and for pointing out its application to Theorem~\ref{ThmA} (private communication). Agliano and Ursini explicitly state the pointed unary biconditional, but do not there introduce its separate directions as the two classes of subsets studied here.

\subsection{Clots and ideals}

Following \cite{AU92}, a \emph{$\top$-clot} is a subset containing $\top$ and closed under every polynomial $q$ satisfying $q(\vec\top)=\top$. The least clot containing $I$ is
\[
\clo I=R_I\top
=\{t(\vec x,\vec y)\mid \vec y\in I,\ t(\vec x,\vec\top)=\top\}.
\]
This is the construction in \cite[Propositions 1.3 and 2.2]{AU92}; our generating relation is the opposite of theirs. In particular, $\clo{\clo I}=\clo I$.

When $\top=0$ is a constant of a variety $V$, recall that a \emph{$V$-ideal term} in $\vec y$ is a term $t(\vec x,\vec y)$ for which $V\models t(\vec x,\vec0)=0$. A $V$-ideal is a subset containing $0$ and closed under these terms, with arbitrary parameters in the $\vec x$ positions \cite[Definition 1.2]{AU92}. Consequently, for subsets containing $0$,
\[
0\text{-normal}\ \Longrightarrow\
0\text{-inductive}\ \Longrightarrow\
0\text{-clot}\ \Longrightarrow\
V\text{-ideal}.
\]
Indeed, the first implication is Theorem~\ref{ThmA}; the second follows from Lemma~\ref{LemA}; and the third follows because each ideal term, after fixing its parameters, is a polynomial taking $\vec0$ to $0$.

For semirings, the ideal-term notion agrees with the usual one: $y_1+y_2$, $xy$, and $yx$ are ideal terms, and expansion of any ideal term as a sum of products proves the converse. Moreover, Theorem~\ref{ThmD} below gives, for non-empty $I$,
\[
\mathcal I_0(I)=I+\overline I,
\qquad
\mathcal D_0(I)=\{x\mid x+y\in I\text{ for some }y\in\overline I\},
\]
where $\overline I$ is the ideal generated by $I$. Thus, when $0\in I$, inductivity is exactly idealhood; when $I$ is an ideal, deductivity is exactly subtractivity. The following example explains why inductivity cannot be replaced by idealhood in the general decomposition.

\begin{example}\label{ExIdealDeductive}
Let $A=\{0,a,b\}$ have a constant $0$ and one unary operation
\[
f(0)=a,\qquad f(a)=b,\qquad f(b)=b,
\]
and let $I=\{0,a\}$. The unary polynomial maps are the identity, $f$, and constants. Lemma~\ref{PropUnary} shows that $I$ is $0$-deductive. It is also a clot: every polynomial taking the all-zero tuple to $0$ is a projection or the constant $0$. Hence it is an ideal, even relative to the variety generated by $A$. Nevertheless it is not inductive, since $f(0)=a\in I$ but $f(a)=b\notin I$. It is not normal either: a congruence identifying $0$ with $a$ must also identify $a=f(0)$ with $b=f(a)$.
\end{example}

\subsection{Positive cones and deductive filters}

There is an exact relational description of inductive sets containing the distinguished element.

\begin{theorem}\label{PropPreorder}
For a subset $I$ containing $\top$, the following are equivalent:
\begin{enumerate}
\item $I$ is $\top$-inductive;
\item $I=P\top$ for some compatible preorder $P$ on $A$.
\end{enumerate}
\end{theorem}

\begin{proof}
If $I$ is inductive, let $P=\bigcup_{n\geqslant0}R_I^n$. This is a compatible preorder. Since $R_I I\subseteq I$ and $\top\in I$, we have $P\top\subseteq I$; the reverse inclusion follows from $I\times\{\top\}\subseteq R_I$. Conversely, if $I=P\top$, then $R_I\subseteq P$ and $R_I I\subseteq P(P\top)=P\top=I$.
\end{proof}

In monoids, the corresponding subsets have been explicitly studied as \emph{positive cones}. Martins-Ferreira and Sobral \cite[Section 2.2]{MS23} characterise the submonoids $M$ which are positive cones of compatible preorders by
\[
u\in M,\quad xy\in M\quad\Longrightarrow\quad xuy\in M.
\]
This is precisely $1$-inductivity. One direction follows by taking $g(z)=xzy$; the other follows by inserting the occurrences of $u$ in an arbitrary unary monoid polynomial one at a time. Also, a $1$-inductive subset containing $1$ is a submonoid, by taking $g(z)=mz$ with $m\in M$. Thus positive cones in \cite{MS23} are exactly the $1$-inductive sets containing $1$, with the opposite orientation for the preorder in Theorem~\ref{PropPreorder}. Section~2.1 of that paper also gives the two-sided condition characterising normal submonoids.

Deductivity has a direct logical instance. If the distinguished element is $1$ and the algebra has an operation $\to$ satisfying $1\to b=b$, then taking $g(z)=z\to b$ in Lemma~\ref{PropUnary} gives
\[
a\in I,\quad a\to b\in I\quad\Longrightarrow\quad b\in I,
\]
the rule of modus ponens.

\begin{theorem}\label{PropHeyting}
In a Heyting algebra, the non-empty $1$-deductive sets are precisely the deductive filters of intuitionistic propositional logic.
\end{theorem}

\begin{proof}
Let $I$ be non-empty and $1$-deductive. Then $1\in I$ and $I$ is closed under modus ponens. If $a\in I$ and $a\leqslant b$, then $a\to b=1$, so $b\in I$. If $a,b\in I$, the identity $a\to(b\to(a\wedge b))=1$ and two applications of modus ponens give $a\wedge b\in I$. Thus $I$ is a lattice filter. Conversely, every lattice filter of a Heyting algebra is the class of $1$ for a congruence, so is $1$-deductive by Theorem~\ref{ThmA}. Lattice filters are exactly the intuitionistic deductive filters; see \cite[Examples 1.9 and 2.6]{Mor23}.
\end{proof}

For Heyting algebras, therefore, the analogy with deductive reasoning in the Introduction gives precisely the familiar notion of a deductive filter.

\subsection{Closure processes}

The operators $\mathcal I_\top$ and $\mathcal D_\top$ are extensive and monotone. Their polynomial descriptions show that they preserve directed unions: any witnessing term involves only finitely many elements of the input set. Consequently, the increasing chains
\[
\indn I0=I\subseteq\indn I1=\ind I\subseteq
\indn I2=\ind{\ind I}\subseteq\cdots
\]
and
\[
\dedn I0=I\subseteq\dedn I1=\ded I\subseteq
\dedn I2=\ded{\ded I}\subseteq\cdots
\]
have unions $\indn I\infty$ and $\dedn I\infty$ which are, respectively, the least inductive and deductive subsets containing $I$. In particular,
\[
\ind{\indn I\infty}=\indn I\infty,
\qquad
\ded{\dedn I\infty}=\dedn I\infty.
\]
The number of steps needed to reach these closures gives the ranks studied below.

The polynomial descriptions also show that, for a homomorphism $h\colon A\to B$ preserving the distinguished elements,
\[
h(\mathcal I_\top(I))\subseteq\mathcal I_\top(h(I)),
\qquad
h(\mathcal D_\top(I))\subseteq\mathcal D_\top(h(I)).
\]
In particular, inverse images of inductive sets are inductive, and inverse images of deductive sets are deductive.

\begin{theorem}\label{ThmC}
Let $I\subseteq A$, and let $R=R_I$. Then
\[
\indn I\infty=\bigcup_{n\geqslant0}R^n I.
\]
For every $n\geqslant0$,
\[
R^n I\subseteq\indn In\subseteq R^{2^n-1}I.
\]
For every $n\geqslant1$,
\begin{equation}\label{EqGeneralDeductionBound}
IR^n\subseteq\dedn In
\subseteq IR(R^\circ R)^{2^{n-1}-1}.
\end{equation}
In particular, if $R$ is symmetric, then
\[
\dedn In\subseteq IR^{2^n-1}\qquad(n\geqslant0),
\qquad
\dedn I\infty=\bigcup_{n\geqslant0}IR^n.
\]
Here a zeroth power of a relation is the diagonal relation.
\end{theorem}

\begin{proof}
We prove the induction bounds by induction on $n$. They hold for $n=0$.
Let $R_n$ be the semicongruence generated by
$\indn In\times\{\top\}$. Assuming the bounds at $n$, we have
\[
\indn I{n+1}=R_n\indn In\subseteq R_nR^{2^n-1}I.
\]
Since $I\times\{\top\}\subseteq R$,
\[
\indn In\times\{\top\}\subseteq
(R^{2^n-1}I)\times\{\top\}\subseteq R^{2^n}.
\]
The relation $R^{2^n}$ is a semicongruence, so $R_n\subseteq R^{2^n}$ and
\[
\indn I{n+1}\subseteq R^{2^{n+1}-1}I.
\]
Also $R\subseteq R_n$, whence
\[
R^{n+1}I\subseteq R\indn In
\subseteq R_n\indn In=\indn I{n+1}.
\]
Taking unions gives the first assertion.

For deduction, put $J_n=\dedn In$ and $S_n=R_{J_n}$. Since
$R\subseteq S_n$, induction gives $IR^n\subseteq J_n$.
For the upper bound define semicongruences
\[
P_0=\Delta_A,\qquad P_{n+1}=P_nP_n^\circ R.
\]
We claim that $J_n\subseteq IP_n$. This holds at $n=0$. If it holds
at $n$ and $j\in J_n$, choose $i\in I$ with $iP_nj$.
As $iR\top$, we get $jP_n^\circ R\top$. Thus
\[
J_n\times\{\top\}\subseteq P_n^\circ R,\qquad
S_n\subseteq P_n^\circ R,
\]
and consequently
\[
J_{n+1}=J_nS_n\subseteq IP_nP_n^\circ R=IP_{n+1}.
\]
The recurrence gives
\[
P_n=R(R^\circ R)^{2^{n-1}-1}\qquad(n\geqslant1),
\]
proving~\eqref{EqGeneralDeductionBound}. If $R=R^\circ$, this is
$R^{2^n-1}$. The symmetric upper bound, together with
$IR^n\subseteq J_n$, gives the stated union formula.
\end{proof}

The symmetry hypothesis holds for $R_I$ in a $\top$-subtractive variety. Indeed, if $s(x,x)=\top$ and $s(x,\top)=x$, then $(a,\top)\in R_I$ gives
\[
(\top,a)=s\bigl((a,a),(a,\top)\bigr)\in R_I.
\]
Thus $R_I^\circ$ contains the generating pairs of $R_I$, and minimality gives $R_I=R_I^\circ$. This observation applies, in particular, to the subtractive and Mal'tsev settings of Subsections~\ref{SubsecSubtractive} and~\ref{SubsecMalcev}.

\begin{remark}\label{RemDeductionCaution}
The alternating bound~\eqref{EqGeneralDeductionBound} does not require symmetry. Without an additional hypothesis, however, its right side cannot be replaced by $IR^{2^n-1}$, nor does the deduction union formula in Theorem~\ref{ThmC} hold, even when $R_I$ is transitive. In the semiring $\mathbb N_0$, take $I=\{3,4\}$. The formula for deduction in Theorem~\ref{ThmD} gives
\[
\mathcal D_0(I)=\{0,1,3,4\},
\qquad
\mathcal D_0^2(I)=\{0,1,2,3,4\}.
\]
Indeed, $\overline I=\{3m+4n\mid m,n\in\mathbb N_0\}$, whereas the ideal generated by $\mathcal D_0(I)$ is all of $\mathbb N_0$. On the other hand, $R_I=\{(x+y,x)\mid y\in\overline I\}$ is transitive, so $IR_I^n=IR_I$ for every $n\geqslant1$. Deduction must therefore be iterated using the relations generated by the successive sets, not just powers of the initial relation.
\end{remark}

\smallskip

\subsection{Ranks and preliminary bounds}\label{SubsecRankPreliminaries}

\begin{definition}
Consider a variety $V$ of algebras over the same signature, and a constant $\top$ in this signature. We say that $V$ has \emph{$\top$-inductive rank} $n$ when $n$ is the smallest non-negative integer such that $\indn{I}{n}=\indn{I}{n+1}$ for every non-empty set $I$ in every algebra $A$ of the variety. When such $n$ does not exist, we say that the variety has $\top$-inductive rank $\infty$. The notion of \emph{$\top$-deductive rank} is defined similarly.
\end{definition}

We write $r_{\mathcal I}(V)$ and $r_{\mathcal D}(V)$ for these ranks when the distinguished constant is understood. We first give preliminary bounds and standard examples. The general rank spectrum theorem is proved in Section~\ref{SecRanks}.

\subsection{Subtractive varieties}\label{SubsecSubtractive}

Following Ursini \cite{Urs94}, a variety having a constant $0$ is \emph{$0$-subtractive} if there is a binary term $s$ satisfying $s(x,x)=0$ and $s(x,0)=x$.
By a \emph{non-trivial variety} we mean a variety that contains algebras having $2$ or more elements.

\begin{theorem}\label{ThmK}
A $0$-subtractive variety has $0$-inductive rank at most $2$ and $0$-deductive rank at most $2$. For a non-empty seed $I$, both second iterates equal its generated $0$-clot. Moreover, $0$-inductive sets are the same as $0$-deductive sets and the non-empty ones are the same as $0$-clots.
\end{theorem}

\begin{proof}
The empty set is both $0$-inductive and $0$-deductive. Suppose $I$ is not empty. Let $R$ denote the semicongruence generated by $I\times\{\top\}$, where $\ast=0$. Let $x\in RRRI$. Then 
$$xRyRzRt, \quad t\in I.$$ Subtracting $(t,\top)\in R$ from $(t,t)$ gives $(\top,t)\in R$. Subtracting this from $(z,t)\in R$ gives $(z,\top)\in R$. Subtracting this from $(z,z)\in R$ gives $(\top,z)\in R$. Continuing this way we eventually get $(x,\top)\in R$. Since $(\top,t)\in R$, we get $(x,t)\in RR$ and hence $x\in RRI$. So $RRRI\subseteq RRI$. This implies that $R^nI\subseteq RRI$ for any $n$. Then, by Theorem~\ref{ThmC}, $\indn{I}{3}\subseteq R^7I\subseteq RRI\subseteq \indn{I}{2}$ and hence $\indn{I}{2}=\indn{I}{3}=RRI$. So the $0$-inductive rank is at most $2$ and $RRI$ is the smallest $0$-inductive set containing $I$.

Now let $t\in IRRR$. Then $$xRyRzRt,\quad x\in I.$$
Subtracting $(x,\top)\in R$ from $(x,y)\in R$ we get $(\top,y)\in R$. This implies $(y,\top)\in R$. Continuing this way we eventually get $(\top,t)\in R$. Since we also have $(x,\top)\in R$, where $x\in I$, we obtain $t\in IRR$. As before, this implies $\dedn{I}{2}=\dedn{I}{3}=IRR$. So the $0$-deductive rank is at most $2$ and $IRR$ is the smallest $0$-deductive set containing $I$. 

Now let $x\in RRI$. Then $xRyRz,$ where $z\in I.$
By successive subtraction as before, we get $x\in R\top$. Conversely, if $x\in R\ast$, then we can deduce $x\in RRI$ by the fact that $I$ is non-empty. So $RRI=R\ast$. We can similarly show that $IRR=R\ast$. Note that $R\ast$ is the smallest clot containing $I$. So the smallest $0$-inductive set containing non-empty $I$ is the same as smallest $0$-deductive set containing $I$ and the same as smallest clot containing $I$. This proves the second part of the theorem.  
\end{proof}

The well known fact that in a $0$-subtractive variety $0$-clots coincide with $0$-normal sets follows from the previous theorem and Theorem~\ref{ThmA}.

\begin{corollary}\label{CorA}
In a $0$-subtractive variety, the following four classes of sets are the same: non-empty $0$-inductive sets, non-empty $0$-deductive sets, $0$-normal sets, $0$-clots.
\end{corollary}

\subsection{Mal'tsev and strongly unital varieties}\label{SubsecMalcev}

A Mal'tsev term improves the bound to one; the same improvement holds in the presence of subtraction and a unital term.

\begin{theorem}\label{ThmMalcevRank}
Consider an algebra $A$ admitting a Mal'tsev term operation, i.e., a term operation $p$ such that the identities $p(x,y,y)=x$ and $p(x,x,y)=y$ hold in the algebra. Then for any set $I\subseteq A$ and any element $\ast\in A$, we have $\ind I=\ded I$ and this set is simultaneously the smallest $\ast$-inductive set containing $I$ and the smallest $\ast$-deductive set containing $I$.
\end{theorem}

\begin{proof}
It is well known that the presence of a Mal'tsev term operation in an algebra forces every reflexive compatible relation on that algebra to be a congruence. Then, by Theorem~\ref{ThmC},
$$\indn{I}{\infty}=I\cup RI=RI=IR=I\cup IR=\dedn{I}{\infty},$$
which completes the proof since $RI=\ind I$ and $IR=\ded I$.
\end{proof}

Recall that a Mal'tsev variety is one where there is a ternary term $p$ satisfying $p(x,y,y)=x$ and $p(x,x,y)=y$. This is in fact equivalent to every reflexive compatible relation being a congruence.

\begin{corollary}\label{CorB}
A Mal'tsev variety has $\ast$-inductive rank as well as $\ast$-deductive rank at most $1$, for any constant $\ast$.  The following three classes of sets are the same in such a variety: non-empty $\ast$-inductive sets, non-empty $\ast$-deductive sets, $\ast$-normal sets.
\end{corollary}

Note that a Mal'tsev variety with a constant $\ast$ is always $\ast$-subtractive: we can define subtraction by setting $s(x,y)=p(x,y,\ast)$. We can also define a J\'onsson--Tarski operation for $\ast$: setting $u(x,y)=p(x,\ast,y)$, we get $u(x,\ast)=x=u(\ast,x)$.
The corollary above in fact generalises to any variety having both of these operations.

\begin{theorem}\label{ThmStrongUnitalRank}
Consider a $0$-subtractive variety $V$ having a J\'onsson--Tarski term for $0$. Then $V$ has $0$-inductive rank as well as $0$-deductive rank at most $1$. The following three classes of sets are the same in such a variety: non-empty $0$-inductive sets, non-empty $0$-deductive sets, and $0$-normal sets.
\end{theorem}

\begin{proof}
Suppose $I$ is non-empty, and put $R=R_I$. By Theorem~\ref{ThmK} and its proof, $\mathcal I_0^2(I)=R^2I=R0=IR^2=\mathcal D_0^2(I)$. Let $u$ be a J\'onsson--Tarski term, so $u(x,0)=x=u(0,x)$. The relation $R$ is symmetric. If $xR0$ and $a\in I$, then $(x,0),(0,a)\in R$, and compatibility with $u$ gives $(x,a)\in R$. Thus $R0\subseteq RI$. Similarly, applying $u$ to $(a,0),(0,x)\in R$ gives $(a,x)\in R$, so $R0\subseteq IR$. Both generation processes are therefore idempotent.

The second part of the theorem follows from Corollary~\ref{CorA}.
\end{proof}

\subsection{Standard examples}\label{SubsecStandardExamples}

The preceding bounds give the familiar normal subsets in groups, rings and modules, together with their one-step generation.

\begin{theorem}\label{PropStandardRanks}
Every non-trivial Mal'tsev variety with a named constant has rank pair $(1,1)$. In particular, both processes generate, in one step, the normal subgroup generated by a non-empty subset of a group, the two-sided ideal generated by a non-empty subset of a ring, and the submodule generated by a non-empty subset of a module. Non-trivial varieties of these algebras have rank pair $(1,1)$.
\end{theorem}

\begin{proof}
Theorem~\ref{ThmMalcevRank} identifies each one-step output with the normal closure. Both ranks are at most one. If $a\ne0$, deduction from $\{a\}$ adjoins $0$. For a Mal'tsev term $p$, the polynomial $g(z)=p(a,z,0)$ satisfies $g(0)=a$ and $g(a)=0$, so induction also adjoins $0$. Thus both bounds are attained. Groups have the Mal'tsev term $xy^{-1}z$, while rings and modules have $x-y+z$. Their normal subsets are respectively normal subgroups, two-sided ideals and submodules.
\end{proof}

For a non-zero ring with identity the variety of unital modules is non-trivial, so its pair is $(1,1)$; for the zero ring both ranks are zero. Heyting algebras likewise have the Mal'tsev term
\[
p(x,y,z)=((x\to y)\to z)\wedge((z\to y)\to x).
\]
Consequently their non-trivial varieties have pair $(1,1)$, with either bound chosen as distinguished constant. With distinguished element $1$, both one-step processes produce the deductive filter generated by the seed, by Theorem~\ref{PropHeyting}.

We now make explicit the semiring case from which the two constructions arose. A semiring has a commutative additive monoid, an associative multiplication distributing over addition, and an absorbing zero; a multiplicative identity may also be included. Ideals are two-sided and contain zero.

\begin{theorem}\label{ThmD}
Let $I$ be a non-empty subset of a semiring $A$, and let $\overline I$ be its generated ideal. Then
\begin{equation}\label{EqSemiringOperators}
\mathcal I_0(I)=I+\overline I,
\qquad
\mathcal D_0(I)=\{x\in A\mid x+y\in I\text{ for some }y\in\overline I\}.
\end{equation}
Induction is idempotent. The $0$-inductive subsets containing $0$ are exactly the ideals, and an ideal is $0$-deductive if and only if it is subtractive. Moreover, deduction from an ideal gives its generated subtractive ideal in one step. The variety of semirings has rank pair $(1,\infty)$, as does each of the varieties obtained by requiring commutative multiplication, a multiplicative identity, or both.
\end{theorem}

\begin{proof}
Expanding a polynomial using distributivity separates its monomials into those containing a variable from $I$ and those involving only parameters. Hence
\[
R_I=\{(x+y,x)\mid x\in A,\ y\in\overline I\}.
\]
Conversely, every element of $\overline I$ is a finite sum of products having at least one factor in $I$, so every displayed pair belongs to $R_I$. This proves~\eqref{EqSemiringOperators}.

Since $I\subseteq I+\overline I\subseteq\overline I$, the sets $I$ and $I+\overline I$ generate the same ideal. Thus
\[
\mathcal I_0^2(I)=I+\overline I+\overline I=I+\overline I.
\]
When $0\in I$, the first formula is $\mathcal I_0(I)=\overline I$, giving the assertion about inductive sets. The seed $\{1\}$ in $\mathbb N_0$ is not inductive, so the inductive rank is one in each stated variety.

If $I$ is an ideal, put $K=\{x\mid x+y\in I\text{ for some }y\in I\}$. Addition and multiplication on either side preserve $K$, by the corresponding properties of $I$. If $x+z,z\in K$, choose $a,b\in I$ with $z+a\in I$ and $x+z+b\in I$. Then
\[
c=z+a+b\in I,\qquad x+c=(x+z+b)+a\in I,
\]
so $x\in K$. Thus $K$ is subtractive. Every subtractive ideal containing $I$ contains $K$, proving minimality.

For deductive rank, take the commutative polynomial semiring
\[
A=\mathbb N_0[x_1,x_2,\ldots],\qquad
I=\{x_j+x_{j+1}\mid j\geqslant0\},\qquad x_0=0.
\]
We claim that, for every $k\geqslant1$,
\begin{equation}\label{EqSemiringInfiniteChain}
\mathcal D_0^k(I)=I\cup\{0,x_1,\ldots,x_{k+1}\}.
\end{equation}
All these polynomials are homogeneous linear polynomials or zero. If
$u+v$ is one of them, non-negativity of coefficients forces $u$ and
$v$ to have the same property. Moreover, for any set $J$ of such
polynomials, the homogeneous degree-one elements, together with zero,
of $\overline J$ are exactly the additive submonoid generated by $J$:
only constant coefficients of multipliers can contribute to a
degree-one polynomial.

Consequently, a residual $u$ with $u+v$ in the current set can only
be zero, an existing adjacent sum, or a single variable. Removing one
endpoint of an adjacent sum requires that endpoint to belong to the
generated ideal. A single variable belongs to this ideal precisely
when it already occurs individually in the current set, since a
non-negative combination of adjacent sums cannot isolate an endpoint.
Initially only $x_1$ occurs individually. The first step therefore
adjoins $0$ and $x_2$, and each subsequent step adjoins exactly the
next variable. This proves~\eqref{EqSemiringInfiniteChain}. The chain
is strictly increasing at every step, so deductive rank is infinite.
The same example works with or without a named multiplicative
identity: after forgetting that constant, it remains available as a
polynomial parameter. It also belongs to the varieties in which
multiplication is not required to be commutative.
\end{proof}

Together with Theorem~\ref{ThmA}, this recovers the characterisation of semiring kernels as subtractive ideals. The infinite deductive rank concerns arbitrary non-empty subsets, not just ideals: deduction is idempotent on ideals, while~\eqref{EqSemiringInfiniteChain} gives an arbitrarily long process even from a subset of a commutative semiring with identity. This is a clear distinction from the Mal'tsev and $0$-subtractive cases, where deduction from every non-empty subset stabilises after at most one and two steps, respectively (Theorems~\ref{ThmMalcevRank} and~\ref{ThmK}).

\section{The general rank spectrum theorem}\label{SecRanks}

We now determine which pairs of ranks are possible. All ranks in this section are taken at the named constant $0$; other named constants are allowed. Write $\operatorname{rk}(V)=(r_{\mathcal I}(V),r_{\mathcal D}(V))$ and $\overline{\mathbb N}_{>0}=\{1,2,\ldots,\infty\}$.

\begin{theorem}[General rank spectrum theorem]\label{ThmCompleteSpectrum}
The rank-pair spectrum of all varieties with a distinguished constant is
\[
\{(0,0),(0,1)\}\ \cup\
\bigl(\overline{\mathbb N}_{>0}\bigr)^2.
\]
Every positive finite pair is realized by a variety generated by a finite algebra in a finite signature.
\end{theorem}

We prove the exclusions first. We then construct diagonal pairs with a permutability property which allows them to be combined with the two families of unequal ranks. The proof of Theorem~\ref{ThmCompleteSpectrum} is completed at the end of the section.

\subsection{The zero-rank exclusions}\label{SubsecRankZero}

\begin{theorem}\label{ThmRankZero}
A variety $V$ has inductive rank zero if and only if each of its terms is identically a projection or a constant term. It has deductive rank zero if and only if $V$ is trivial. Consequently, the only possible pairs involving zero are $(0,0)$ and $(0,1)$, and both occur.
\end{theorem}

\begin{proof}
If every term is a projection or a constant term, then
\[
R_I=\Delta_A\cup(I\times\{0\}),\qquad
\mathcal I_0(I)=I,\qquad \mathcal D_0(I)=I\cup\{0\}
\]
for every non-empty $I\subseteq A$. Hence the inductive rank is zero and the deductive rank is at most one.

Conversely, suppose induction has rank zero. Given a term $t(x_1,\ldots,x_n)$, consider the free $V$-algebra on these variables and the subset
\[
I=\{x_1,\ldots,x_n,t(0,\ldots,0)\}.
\]
Then $t(x_1,\ldots,x_n)\in\mathcal I_0(I)=I$. Equality in the free algebra says that $V$ satisfies either $t(x_1,\ldots,x_n)=x_i$ for some $i$, or $t(x_1,\ldots,x_n)=t(0,\ldots,0)$. These are the asserted alternatives.

For every $a\in A$, the set $\mathcal D_0(\{a\})$ contains $0$. Thus deductive rank zero forces $a=0$ in every algebra, and the variety is trivial. Conversely, the trivial variety has pair $(0,0)$. Every non-trivial variety covered by the first paragraph has pair $(0,1)$; pointed sets give such a variety.
\end{proof}

\subsection{Permutability and diagonal ranks}\label{SubsecPermutableWitnesses}

Recall that a variety is \emph{congruence $n$-permutable}, where
$n\geqslant2$, if the two alternating composites of any two congruences,
with $n$ factors each, are equal. Equivalently, it has ternary
Hagemann--Mitschke terms $p_1,\ldots,p_{n-1}$ satisfying
\begin{equation}\label{EqHMIdentities}
\begin{gathered}
p_1(x,y,y)=x,\qquad p_{n-1}(x,x,y)=y,\\
p_i(x,x,y)=p_{i+1}(x,y,y)\qquad(1\leqslant i<n-1).
\end{gathered}
\end{equation}
See \cite{HM73}.

\begin{theorem}\label{ThmPermutableRank}
Let $\mathcal V$ be a congruence $n$-permutable variety with a named
constant $0$. For every non-empty subset $I$ of an algebra in
$\mathcal V$,
\[
\mathcal I_0^{\,n-1}(I)=\mathcal D_0^{\,n-1}(I)
=[0]_{\operatorname{Cg}(I\times\{0\})}.
\]
Consequently, both ranks of $\mathcal V$ are at most $n-1$.
\end{theorem}

\begin{proof}
Put $m=n-1$ and $R=R_I$. Hagemann's theorem, in the form of
\cite[Theorem~1.1]{JRV14}, gives
\[
R^\circ\subseteq R^m,\qquad R^{m+1}\subseteq R^m.
\]
Reflexivity implies $R^j=R^m$ for every $j\geqslant m$. Thus $C=R^m$
is transitive, and it is symmetric since
\[
C^\circ=(R^\circ)^m\subseteq R^{m^2}=C.
\]
It is also reflexive and compatible, and hence a congruence.
Every congruence containing $I\times\{0\}$ contains $R$ and its
powers, while $C$ contains those generating pairs. Therefore
\begin{equation}\label{EqHagemannStarPower}
R_I^{\,m}=\operatorname{Cg}(I\times\{0\}).
\end{equation}
Write $N=[0]_C$. Since $\varnothing\ne I\subseteq N$,
\[
R^mI=CI=N=IC=IR^m.
\]
Each fixed-$R$ step can be performed in the corresponding generation
process: the relation at every subsequent stage contains $R$.
Thus $N$ is contained in both $m$th iterates. Conversely, whenever
the current set lies in $N$, its generated relation lies in $C$ and
either one-step operator stays in $N$. Both processes consequently
remain inside $N$, proving the equality.
\end{proof}

We need finite witnesses attaining both bounds simultaneously. The
following explicit construction will also provide the factors used
in the product argument.

\begin{theorem}\label{ThmHMSharpRanks}
For every $m\geqslant1$, there is a finite algebra
$B_m$ such that
\[
\mathcal P_m=\operatorname{HSP}(B_m)
\]
is congruence $(m+1)$-permutable and has exact rank pair $(m,m)$.
Its positive-arity basic operations preserve its distinguished $0$.
\end{theorem}

\begin{proof}
On $X_m=\{0,1,\ldots,m\}$ define ternary operations
\begin{equation}\label{EqHMFiniteOperations}
p_i(x,y,z)=\max\bigl\{
\min(x,m-i+\delta_{y,z}),\
\min(z,i-1+\delta_{x,y}),\
\min(x,y,z)\bigr\},\qquad 1\leqslant i\leqslant m,
\end{equation}
where $\delta_{a,b}$ is $1$ when $a=b$ and $0$ otherwise. Minimum
and maximum describe these operation tables and are not additional
basic operations. For $0\leqslant j\leqslant m$, put
\[
h_j(x,z)=\max\{\min(x,m-j),\min(z,j),\min(x,z)\}.
\]
Direct substitution gives
\[
p_i(x,x,z)=h_i(x,z),\qquad p_i(x,z,z)=h_{i-1}(x,z),
\]
with $h_0(x,z)=x$ and $h_m(x,z)=z$. Thus the operations satisfy
\eqref{EqHMIdentities} for $n=m+1$. In particular they are idempotent.

The reflexive relation
\[
Q=\{(a,b)\in X_m^2\mid a\leqslant b+1\}
\]
is compatible with all $p_i$. Indeed, if each coordinate of a first
triple is at most the corresponding coordinate of a second triple
plus $1$, the same is true of the two thresholds involving $\delta$
in \eqref{EqHMFiniteOperations}. Taking minima and maxima preserves
this one-sided inequality. The converse relation $Q^\circ$ is
compatible as well.

The integers in $X_m$ are numerical labels; the distinguished constant
may be interpreted as any of them. Let $\mathbf A_m^{(t)}$ denote the
algebra with the displayed operations and with distinguished constant
interpreted as $t$.

First take $t=m$ and the numerical seed $\{0\}$. Since
$X_m\times\{m\}\subseteq Q$, we have $R_J\subseteq Q$ for every
subset $J$. Consequently induction can increase the largest label
present by at most one. For $0\leqslant k<m$, the unary polynomial
\[
u_k(t)=p_{m-k}(m,t,0)
\]
satisfies $u_k(0)=k+1$ and $u_k(m)=k$. It therefore inductively
adjoins $k+1$ whenever the labels $0$ and $k$ are present. Hence,
with the subscript $0$ here referring to the distinguished constant
of $\mathbf A_m^{(m)}$,
\begin{equation}\label{EqHMInductionChain}
\mathcal I_0^{\,k}(\{0\})=\{0,1,\ldots,k\}
\qquad(0\leqslant k\leqslant m).
\end{equation}

Next take $t=1$, again with numerical seed $\{0\}$. Now
$X_m\times\{1\}\subseteq Q^\circ$, so deduction can also increase
the largest label by at most one. Its first step adjoins the
distinguished constant, whose label is $1$. For $1\leqslant k<m$,
\[
v_k(t)=p_{m-k}(m,t,1)
\]
satisfies $v_k(0)=k$ and $v_k(1)=k+1$, giving the next deductive
step. Thus, in $\mathbf A_m^{(1)}$,
\begin{equation}\label{EqHMDeductionChain}
\mathcal D_0^{\,k}(\{0\})=\{0,1,\ldots,k\}
\qquad(0\leqslant k\leqslant m).
\end{equation}
For $m=1$, only the first step is needed.

Set
\[
B_m=\mathbf A_m^{(m)}\times\mathbf A_m^{(1)}.
\]
Its two factors are homomorphic images and supply the respective
lower bounds in $\operatorname{HSP}(B_m)$.
The identities \eqref{EqHMIdentities} hold in this variety, so
Theorem~\ref{ThmPermutableRank} supplies both upper bounds $m$.
This proves the exact pair. The generator has $(m+1)^2$ elements;
idempotence of its operations shows that they preserve its
distinguished constant $(m,1)$.
\end{proof}

\subsection{Independent products with a permutable factor}\label{SubsecIndependentProducts}

Let $\mathcal V$ and $\mathcal W$ have signatures renamed to be
disjoint apart from $0$. Expand each by interpreting the other's
positive-arity operations as first projections and its additional
constants as $0$. Add a binary operation $p$, interpreted as the
first projection in $\mathcal V$ and the second in $\mathcal W$.
In a product of the expanded algebras $A$ and $B$ we therefore have
\begin{equation}\label{EqIndependentProductOperations}
\begin{aligned}
p((a,b),(a',b'))&=(a,b'),\\
f((a_1,b_1),\ldots,(a_t,b_t))
 &=\bigl(f_A(a_1,\ldots,a_t),b_1\bigr),\quad f\in\mathcal V,\\
g((a_1,b_1),\ldots,(a_t,b_t))
 &=\bigl(a_1,g_B(b_1,\ldots,b_t)\bigr),\quad g\in\mathcal W.
\end{aligned}
\end{equation}
Its distinguished constant is $(0_A,0_B)$; an additional constant
$c$ of $\mathcal V$ is $(c_A,0_B)$, and similarly for $\mathcal W$.

The class of these products is denoted by
$\mathcal V\boxtimes\mathcal W$. This is an independent join in
the sense of \cite[Theorem~1]{GLP69}. The selector also proves
directly that it is a variety: every subalgebra is the product of
its projections, and the same holds for every congruence, viewed
as a subalgebra of $A^2\times B^2$. The latter projections are
congruences, so quotients retain the product form; arbitrary
products do so by regrouping their factors. Conceptually this
presents the product category $\mathcal V\times\mathcal W$ in one
sort, with underlying sets having the form $A\times B$ where $A\in \mathcal V$ and $B\in \mathcal W$.

Taking the other factor to be trivial includes either expanded
component variety. The added projections and duplicated constants
do not change its polynomial operations. Thus each product rank
is at least the corresponding rank of either factor. Moreover,
if $\mathcal V=\operatorname{HSP}(A)$ and
$\mathcal W=\operatorname{HSP}(B)$, then
\begin{equation}\label{EqIndependentProductGenerator}
\mathcal V\boxtimes\mathcal W
=\operatorname{HSP}(A\boxtimes B).
\end{equation}
Indeed, let $\widehat A$ and $\widehat B$ be the expansions of $A$ and
$B$ to the common signature just described. The two projections from
$A\boxtimes B$ are surjective homomorphisms onto $\widehat A$ and
$\widehat B$. Therefore $\operatorname{HSP}(A\boxtimes B)$ contains
the varieties generated by both expansions, and hence their join.
By \cite[Theorem~1]{GLP69}, this independent join is precisely
$\mathcal V\boxtimes\mathcal W$. Conversely, the latter is a variety
containing $A\boxtimes B$, so it contains its generated variety.
This proves~\eqref{EqIndependentProductGenerator}. In particular, the
product of two finite generators is a finite generator for the
independent product, and finite signatures remain finite. 

\begin{theorem}\label{ThmPermutableProductRanks}
If $\mathcal V$ has rank pair $(i,d)$ and $m\geqslant1$, then
\[
\bigl(r_{\mathcal I}(\mathcal V\boxtimes\mathcal P_m),
r_{\mathcal D}(\mathcal V\boxtimes\mathcal P_m)\bigr)
=\bigl(\max\{i,m\},\max\{d,m\}\bigr).
\]
Here $i$ and $d$ may be infinite, and $\mathcal P_m=\operatorname{HSP}(B_m)$
is the congruence $(m+1)$-permutable variety of exact pair $(m,m)$
constructed in Theorem~\ref{ThmHMSharpRanks}.
More generally, if $\mathcal W$
is congruence $(m+1)$-permutable, $\varnothing\ne S\subseteq A\times B$,
$P=\pi_A S$, $Q=\pi_B S$ and
\[
L=[0_B]_{\operatorname{Cg}_B(Q\times\{0_B\})},
\]
then, for either one-step operator $F=\mathcal I_0$ or
$F=\mathcal D_0$,
\begin{equation}\label{EqProductRectangularization}
F^k(S)=F_A^k(P)\times L\qquad(k\geqslant m).
\end{equation}
\end{theorem}

\begin{proof}
For every non-empty subset $J\subseteq A\times B$, the selector,
applied coordinatewise to pairs, makes $R_J$ a product of its two
coordinate projections. Projection commutes with generation of a
subalgebra, so these projections are precisely the generated star
relations in the factors. Consequently
\begin{equation}\label{EqProductStarRelations}
R_J=R_{\pi_AJ}\times R_{\pi_BJ}.
\end{equation}
For one application of $F$, projecting therefore gives inclusions
into the corresponding factor stages. Conversely, lift an anchor
in one projection to a member of $J$ and keep the other coordinate
unchanged, using reflexivity. Iterating gives
\[
\pi_A F^j(S)=F_A^j(P),\qquad
\pi_B F^j(S)=F_B^j(Q).
\]

By \eqref{EqHagemannStarPower},
\[
R_Q^m=\operatorname{Cg}_B(Q\times\{0_B\}).
\]
Any point of $Q$ can thus be joined to any point of $L$ by a
length-$m$ $R_Q$-path in either direction. Reflexivity pads such
a path to any length $k\geqslant m$.

Take $a\in F_A^k(P)$ and follow its successive anchors backwards
to obtain $a_0,\ldots,a_k=a$, with $a_0\in P$ and step $j$ using
$R_{F_A^{j-1}(P)}$ in the direction appropriate to $F$. Choose
$b_0$ with $(a_0,b_0)\in S$. For any $b\in L$, choose a padded
length-$k$ $R_Q$-path from $b_0$ to $b$ in the same direction.
At each stage $R_Q$ is contained in the relation generated by the
current second projection. Formula~\eqref{EqProductStarRelations}
therefore permits the two paths to be followed simultaneously
through the actual successive subsets $F^j(S)$. Their endpoint
is $(a,b)$, proving the inclusion from right to left in
\eqref{EqProductRectangularization}. The converse follows from
the projection identities and Theorem~\ref{ThmPermutableRank},
which gives $F_B^k(Q)=L$ for $k\geqslant m$.

For a finite inductive rank $i$ of $\mathcal V$, apply
\eqref{EqProductRectangularization} at $k=\max\{i,m\}$ and at
$k+1$ to obtain the corresponding upper bound for the product.
The deductive argument is identical; an infinite upper bound
requires no argument. For $\mathcal W=\mathcal P_m$, the two
component varieties and Theorem~\ref{ThmHMSharpRanks} give
matching lower bounds in both coordinates.
\end{proof}

\subsection{Prescribing inductive rank with deductive rank one}
\label{SubsecAnnihilatingSpectrum}

In this subsection the distinguished element $0$ is the \emph{greatest}
element of a meet-semilattice. This convention makes deduction especially
simple.

\begin{lemma}\label{LemMeetDeductiveClosure}
In an expansion of a meet-semilattice with greatest element $0$ by
order-preserving operations, $\mathcal D_0(I)=\mathord\uparrow I$ for every
non-empty subset $I$.
\end{lemma}

\begin{proof}
Every polynomial is order-preserving, so
$q(\vec a)\leqslant q(\vec0)$. A deductive step therefore moves upwards
from an element of $I$. Conversely, if $a\in I$ and $a\leqslant b$, the
polynomial $q(x)=x\wedge b$ has $q(a)=a$ and $q(0)=b$.
\end{proof}

For $N\geqslant1$, let $\mathcal Q_N$ be the variety of meet-semilattices
with greatest element $0$ and unary operations $f_1,\ldots,f_N$ satisfying
\begin{equation}\label{EqAnnihilatingSpectrumAxioms}
\begin{gathered}
f_i(x\wedge y)=f_i(x)\wedge f_i(y),\qquad
f_i(x)\wedge x=f_i(x),\\
f_i(f_j(x))=f_1(f_1(0))\qquad(1\leqslant i,j\leqslant N).
\end{gathered}
\end{equation}
Write $c=f_1(f_1(0))$ and $c_i=f_i(0)$. Contractivity gives
$c=f_i(f_j(x))\leqslant x$, so $c$ is least and $f_i(c)=c$.
Thus no additional basic constant is required. Let $\mathcal Q_0$ be the
variety of meet-semilattices with greatest element $0$.

\begin{theorem}\label{ThmAnnihilatingSpectrum}
For every $N\geqslant0$, the rank pair of $\mathcal Q_N$ is $(N+1,1)$.
This same pair is attained by a variety generated by one finite algebra:
by a two-element chain when $N=0$, and by a chain $L_N\in\mathcal Q_N$
with $N+4$ elements when $N\geqslant1$.
\end{theorem}

\begin{proof}
Write $T=\mathcal I_0$. Lemma~\ref{LemMeetDeductiveClosure} gives
deductive rank at most $1$, and Theorem~\ref{ThmRankZero} makes it exactly
$1$ in every non-trivial subvariety. For $N=0$, $T(I)$ is the closure of
$I$ under non-empty finite meets. Its rank is at most $1$; in the square
of a two-element chain, two incomparable elements acquire their meet in
one step. This also proves the finite-generator assertion for $N=0$.
Suppose henceforth that $N\geqslant1$.

Distributing the $f_i$ over meets and eliminating their composites puts
every non-constant polynomial into a meet of a parameter, variables and
factors $f_i(x_j)$. Consequently a one-step value from a non-empty $J$
has the form
\begin{equation}\label{EqAnnihilatingSpectrumStep}
\begin{gathered}
b\wedge a_1\wedge\cdots\wedge a_k
 \wedge f_{i_1}(u_1)\wedge\cdots\wedge f_{i_\ell}(u_\ell),\\
b,a_j,u_j\in J,\qquad b\leqslant c_{i_t}\quad(1\leqslant t\leqslant\ell).
\end{gathered}
\end{equation}
Empty factor lists are allowed. Indeed, insert the factor
$b=q(\vec0)$ into $q(\vec a)$, which leaves its value unchanged.
Conversely, use $b$ as a parameter in the displayed expression; the
inequalities ensure that its all-zero value is $b$.
Constant polynomials add nothing.

First take a non-empty meet-closed set $S$. The set
$K=\{x\mid f_i(x)=c\text{ for all }i\}$ is a downset containing every
$f_j(A)$. Formula~\eqref{EqAnnihilatingSpectrumStep} shows that all
iterates from $S$ lie in $S\cup K$. In any derivation of a value other
than $c$, an input under an $f_i$ must therefore belong to $S$: an input
outside $S$ belongs to $K$ and would make the whole meet equal to $c$.
Flattening the derivation now expresses its value as an \emph{admissible
word}
\[
b\wedge f_{i_1}(a_1)\wedge\cdots\wedge f_{i_\ell}(a_\ell),
\qquad b,a_t\in S,
\]
where the prefix preceding $f_{i_t}(a_t)$ is at most $c_{i_t}$.
For completeness, the meet of two such words is admissible: meet their
initial elements in $S$ and concatenate their factor lists. Each new
prefix is below its original prefix. Formula~\eqref{EqAnnihilatingSpectrumStep}
then permits the remaining positive factors, with inputs in $S$, to be
appended. This proves the assertion by induction on derivations.

Combine all occurrences of an index $i$ at its first occurrence, using
$\bigwedge_t f_i(a_t)=f_i(\bigwedge_t a_t)$ and meet closure of $S$.
Moving these factors earlier only decreases intervening prefixes, so
admissibility is preserved. The resulting word has at most $N$ factors,
with distinct indices, and can be formed in at most $N$ successive
inductive steps from $S$.
The element $c$ causes no exception: if $b\in S$ satisfies
$b\leqslant c_i$, then $f_i(b)\leqslant f_i(c_i)=c$, and the polynomial
$b\wedge f_i(z)$ produces $c$ from $b$ in one step. If no such $b,i$
exist,~\eqref{EqAnnihilatingSpectrumStep} gives $T(S)=S$.
Thus $T^N(S)$ is the full inductive closure of $S$.

For arbitrary non-empty $I$, its non-empty finite-meet closure $S$
satisfies $I\subseteq S\subseteq T(I)$. Every inductive set is
meet-closed, so $S$ and $I$ have the same inductive closure. Hence that
closure is reached in $T^{N+1}(I)$, proving the uniform bound $N+1$.

For sharpness, take the chain
\[
L_N:\quad c<d<c_N<\cdots<c_1<a<0,
\]
put $c_{N+1}=d$, and define
\[
f_i(0)=c_i,\qquad f_i(a)=c_{i+1},\qquad
f_i(x)=c\quad(x\leqslant c_1).
\]
These maps are contractive and meet-preserving, and all their composites
are constantly $c$. Thus $L_N\in\mathcal Q_N$.
In $L_N^2$ start with
$I=\{(c_1,0),(0,c_1),(a,a)\}$.
No original anchor lies below any $(c_i,c_i)$, so $T(I)$ is exactly the
non-empty finite-meet closure of $I$; it first supplies $(c_1,c_1)$.
Thereafter the polynomial
$q_i(z)=(c_i,c_i)\wedge f_i(z)$, applied to the original input $(a,a)$,
supplies $(c_{i+1},c_{i+1})$ at step $i+1$.
In particular, $(d,d)\in T^{N+1}(I)$.

There is no shortcut. For $1\leqslant k\leqslant N+1$, every point of
$T^k(I)$ whose two coordinates exceed $c$ has both coordinates at least
$c_k$. The case $k=1$ follows from the meet-closure description.
In~\eqref{EqAnnihilatingSpectrumStep}, a value with both coordinates
above $c$ has an anchor and zero-exponent inputs with that same property.
At step $k+1$ the anchor therefore lies above $(c_k,c_k)$.
A permitted factor $f_i$ requires
$(c_k,c_k)\leqslant b\leqslant(c_i,c_i)$, hence $i\leqslant k$.
Each of its coordinates above $c$ is $c_i$ or $c_{i+1}$, and so is at
least $c_{k+1}$. Taking the meet proves the assertion.
Thus $(d,d)\notin T^N(I)$. The same square witnesses this lower bound
in the variety generated by $L_N$, while its inclusion in $\mathcal Q_N$
gives the uniform upper bounds. Both asserted rank pairs follow.
\end{proof}

In what follows, $\mathcal F_N$ denotes the variety generated by $L_N$
for $N\geqslant1$, and $\mathcal F_0$ the variety generated by the
two-element meet-semilattice with greatest element $0$.
Thus $\mathcal F_N$ is generated by one finite algebra and has rank pair
$(N+1,1)$; no equality $\mathcal F_N=\mathcal Q_N$ is needed.

\begin{theorem}\label{ThmInfiniteInduction}
The variety $\mathcal S_\infty$ of meet-semilattices with greatest
element $0$ and a contractive meet-endomorphism $f$ has rank pair
$(\infty,1)$. The same pair is attained by the variety generated by
$\mathbb Z_{\leqslant0}$ with meet minimum and $f(x)=x-1$.
\end{theorem}

\begin{proof}
Deduction is upward closure by Lemma~\ref{LemMeetDeductiveClosure}.
In the displayed algebra put
\[
J_M=\{-1,-2,\ldots,-M\}\qquad(M\geqslant1).
\]
Every polynomial is a meet of a parameter and factors
$f^k(x_j)$, with $k\geqslant0$. If its all-zero value is an anchor
$b\in J_M$, every exponent $k$ that occurs satisfies
$b\leqslant f^k(0)=-k$, hence $k\leqslant M$. After inserting the
anchor into the meet, every factor evaluated on $J_M$ is at least
$-2M$, and the resulting value is at most $-1$. Thus
$\mathcal I_0(J_M)\subseteq J_{2M}$.
Conversely, $q(x)=f^M(x)$ has $q(0)=-M\in J_M$ and takes the values
$-M-1,\ldots,-2M$ on $J_M$. Extensivity gives
$\mathcal I_0(J_M)=J_{2M}$. Consequently
\[
\mathcal I_0^k(\{-1\})=J_{2^k}\qquad(k\geqslant0),
\]
and these sets increase strictly at every step. This gives infinite
inductive rank already in the variety generated by this algebra.
Non-triviality and Theorem~\ref{ThmRankZero} give deductive rank exactly $1$.
\end{proof}

\subsection{Varieties of rank pair \texorpdfstring{$(1,d)$}{(1,d)}}\label{SubsecCornishRanks}

A \emph{Cornish algebra} is a bounded distributive lattice expanded by
unary operations, each of which is either a lattice endomorphism
preserving the bounds or a lattice dual endomorphism interchanging
them; see \cite[Definition~5.1]{DaveyGair16}. The distinguished constant
in this subsection is the lattice bottom. A word in the additional
unary operations will be called \emph{even} or \emph{odd} according to
the parity of the number of dual endomorphisms occurring in it. The
empty word is even.

\begin{lemma}\label{LemCornishInduction}
Every variety of Cornish algebras has inductive rank at most one.
More precisely, for a non-empty subset $I$ of a Cornish algebra, let
$J$ be the lattice ideal generated by the even-word images of $I$, and
let $F$ be the lattice filter generated by its odd-word images. Then
\begin{equation}\label{EqCornishRelation}
sR_It\quad\Longleftrightarrow\quad
s\leqslant t\vee j\ \text{ and }\ t\wedge h\leqslant s
\quad\text{for some }j\in J,\ h\in F.
\end{equation}
In particular,
\begin{equation}\label{EqCornishDeduction}
\mathcal D_0(I)=
\{x\mid x\wedge h\leqslant i\text{ for some }i\in I, h\in F\}.
\end{equation}
\end{lemma}

\begin{proof}
The pairs $(j,0)$, for $j\in J$, and $(h,1)$, for $h\in F$, belong
to $R_I$. Indeed, applying words to its generators gives the required
pairs on the ideal and filter generators; finite joins and meets,
followed by meeting or joining with diagonal parameters, give all the
indicated pairs. If the inequalities in
\eqref{EqCornishRelation} hold, the lattice polynomial
\[
q(X,Y)=(s\wedge X)\vee(t\wedge Y)\vee(s\wedge t)
\]
has $q(j,h)=s$ and $q(0,1)=t$. This proves one inclusion. Conversely,
the relation on the right of \eqref{EqCornishRelation} is reflexive
and lattice-compatible. Each basic endomorphism preserves $J$ and
$F$, while each basic dual endomorphism maps $J$ into $F$ and $F$ into
$J$. The same relation is therefore compatible with every basic
operation. It contains $I\times\{0\}$, proving equality. It is also
transitive: for two consecutive pairs, join their ideal witnesses and
meet their filter witnesses.

If $x\in\mathcal I_0(I)$, the first inequality with an anchor $i\in I$
gives $x\leqslant i\vee j$, and hence $x\in J$. Every even word maps
$J$ into $J$, and every odd word maps it into $F$. Thus replacing $I$
by $\mathcal I_0(I)$ changes neither $J$ nor $F$. Consequently
$R_{\mathcal I_0(I)}=R_I$, and transitivity gives
$\mathcal I_0^2(I)=\mathcal I_0(I)$. Finally, in the criterion for
$iR_Ix$, the first inequality can always be satisfied by taking $j=i$.
This gives \eqref{EqCornishDeduction}.
\end{proof}

For a non-trivial Cornish variety, inductive rank is exactly one by
Theorem~\ref{ThmRankZero}, since its lattice operations cannot be
projections or constants. There is no finite uniform bound on
deductive rank under the assumptions of Lemma~\ref{LemCornishInduction}:
it can be any positive integer or infinity, as
Theorem~\ref{ThmSpecialCornishSpectrum} below will show.

For $m\geqslant2$, put $X_m=\{0,\ldots,m\}$ and define maps
\[
\sigma(h)=\max\{h-1,0\},\qquad
\tau_j(h)=
\begin{cases}0,&h=j,\\ h,&h\ne j,\end{cases}
\quad 1\leqslant j\leqslant m.
\]
Let $A_m$ be the powerset lattice of $X_m$, with its bounds, expanded
by the operations
\[
f(U)=X_m\setminus\sigma^{-1}(U),\qquad
e_j(U)=\tau_j^{-1}(U)\quad(1\leqslant j\leqslant m),
\]
and put $\mathcal C_m=\operatorname{HSP}(A_m)$. Thus $f$ is a dual
endomorphism and the $e_j$ are endomorphisms. Boolean complementation
is not a basic operation. The point $0\in X_m$, which we call the
\emph{root}, is distinct from the algebraic constant
$0_{A_m}=\varnothing$. The maps $\tau_j$ mark the individual heights:
each collapses just one height to the root.

\begin{theorem}\label{ThmCornishRanks}
For every $m\geqslant2$,
\[
\bigl(r_{\mathcal I}(\mathcal C_m),
      r_{\mathcal D}(\mathcal C_m)\bigr)=(1,m-1).
\]
In particular, every pair $(1,d)$ with $d\geqslant1$ is realized by a
variety generated by a finite algebra.
\end{theorem}

\begin{proof}
Lemma~\ref{LemCornishInduction} gives the inductive upper bound. It is
strict on $\{1\}$, since $f(0)=1$ and $f(1)=0$. We prove the deductive
bound using finite duality, keeping track of the homomorphic images
as well as the subalgebras of powers of $A_m$.

\emph{The finite dual spaces.}
The finite case of Cornish duality represents an algebra as the
lattice $\operatorname{Up}(P)$ of upsets of a finite poset $P$; the
operators $f,e_j$ are represented by an order-reversing map $\sigma$
and order-preserving maps $\tau_j$, respectively. The formulas for
their action on upsets are the same inverse-image formulas as above.
This is the framework of \cite[Theorem~5.2]{DaveyGair16}; we explain
the representation details needed here.

Every finite $B\in\mathcal C_m$ is a quotient of a subalgebra $C$ of
a finite power of $A_m$: use the algebra of term functions on $A_m$
on a finite generating set of $B$. The finite lattice dual of $C$
is an equivariant monotone image of a disjoint union of copies of
the discrete space $X_m$. Indeed, a prime filter of $C$ extends to
the finite Boolean lattice containing it, where it is a coordinate
evaluation. The quotient from $C$ to $B$ identifies the dual of $B$
with an invariant induced subposet of the dual of $C$. These are the
finite instances of the embedding--surjection and product--disjoint
union correspondences in \cite[Lemmas~2.7 and~2.10]{DaveyGair16}.

It follows that the dual $P$ of $B$ has the following properties.
Every point $p$ has a root
\[
\rho(p)=\sigma^m(p)=\sigma^{m+1}(p),
\]
fixed by all the maps. A nonroot point has a unique height
$j\in\{1,\ldots,m\}$: $\tau_j(p)=\rho(p)\ne p$, while all other
markers fix $p$. Its image under $\sigma$ has height $j-1$, unless
it is already the root. To verify survival under identifications,
take a representative at height $j$. Its image is fixed by the other
markers, and $\tau_j$ sends it to its root. If $\tau_j$ also fixes
it, it is a root. Otherwise this characterizes height $j$ uniquely;
representatives of different heights cannot have the same nonroot
image. The assertion about $\sigma$ follows from the same
representatives. Invariant subposets inherit all these properties.

The map $\rho$ is both order-preserving and order-reversing, since
its two displayed expressions have opposite parities. Thus comparable
points have equal roots. Moreover, if $p\leqslant q$ have different
nonzero heights, applying the marker of the height of $q$ gives
\begin{equation}\label{EqCornishHeightSeparation}
p\leqslant\rho(p)=\rho(q).
\end{equation}

\emph{Deductive profiles.}
For a subset $z\subseteq P$, let $L(z)$ be the union of $w^{-1}(z)$
over all words in the dual maps having an odd number of occurrences
of $\sigma$. For a non-empty seed $I\subseteq\operatorname{Up}(P)$,
put $z_0=\bigcup I$. The filter in
\eqref{EqCornishDeduction} is generated by $P\setminus L(z_0)$.
Consequently
\begin{equation}\label{EqCornishUpsetDeduction}
\mathcal D_0(I)=
\{W\in\operatorname{Up}(P)\mid W\subseteq U\cup L(z_0)
                              \text{ for some }U\in I\}.
\end{equation}
For an upset $U$ and an arbitrary subset $L$ of $P$, write $c_L(U)$
for the greatest upset contained in $U\cup L$. If $L\subseteq M$,
then
\begin{equation}\label{EqCornishAbsorption}
c_M(c_L(U))=c_M(U).
\end{equation}
This follows from $U\subseteq c_L(U)\subseteq U\cup L\subseteq U\cup M$.

Put $S_k=\mathcal D_0^k(I)$, $z_k=\bigcup S_k$, and $L_k=L(z_k)$.
The $L_k$ increase. Equations~\eqref{EqCornishUpsetDeduction}
and~\eqref{EqCornishAbsorption} give, for every $k\geqslant1$,
\begin{equation}\label{EqCornishAnchorProfiles}
S_k=\bigcup_{U\in I}
\{W\in\operatorname{Up}(P)\mid W\subseteq c_{L_{k-1}}(U)\}.
\end{equation}
Thus all iterations can be computed using the original anchors.
In particular, $p\in c_L(U)$ precisely when every point above $p$
belongs to $U\cup L$. Formula~\eqref{EqCornishAnchorProfiles} retains
the correlations between the different root components of an anchor.

\emph{The upper bound.}
Every odd dual word sends a nonroot point to a strictly smaller height
or to its root. An effective marker sends it to the root, and otherwise
the odd number of occurrences of $\sigma$ lowers its height. All words
fix roots. Therefore the root membership of $z_k$ never changes:
an absent root belongs neither to an anchor nor to $L_k$ and cannot
be adjoined by \eqref{EqCornishUpsetDeduction}.

If a root belongs to $z_0$, all points having that root belong to
$L_0$, by a sufficiently large odd power of $\sigma$. Since order
comparisons do not cross root components, this whole component is
filled at the first step and remains fixed. Now consider a component
whose root $r$ is absent from $z_0$. Its root and every point below
the root remain absent. If $p\nleqslant r$, every point above $p$
has the same height as $p$, by \eqref{EqCornishHeightSeparation};
a root or a point below the root cannot be above $p$ either.

On the remaining points of height one, $L_k$ is empty for every $k$,
so their union profile is unchanged from $z_0$. At height $j\geqslant2$,
$L_k$ depends only on union profiles at smaller heights. By induction
on $j$, those profiles have stabilized by step $j-2$. Hence the free
region at height $j$ is fixed by $L_{j-2}$, and
\eqref{EqCornishAnchorProfiles}, whose upper-cone test uses only that
height, shows that its union profile is fixed by $z_{j-1}$. Points
below the root have constant zero profiles throughout. Thus $L_{m-2}$
is already the final free region, including the previously treated
root components. Equation~\eqref{EqCornishAnchorProfiles} now gives
$S_m=S_{m-1}$.

This proves the bound for every finite $B\in\mathcal C_m$. The
variety is locally finite, because its finite-signature generator is
finite. A witness to failure of
$\mathcal D_0^m(I)=\mathcal D_0^{m-1}(I)$ in an arbitrary member uses
only finitely many original seed elements and polynomial parameters.
Their generated subalgebra is finite, where the bound just proved
applies. Hence $r_{\mathcal D}(\mathcal C_m)\leqslant m-1$ uniformly.

\emph{Sharpness.}
For $m=2$, nontriviality gives the required lower bound one. Suppose
$m\geqslant3$ and put $n=m-2$. Take $n$ branches, indexed by
$j=0,\ldots,n-1$, sharing a root $r$. Branch $j$ has one point at each
height $1,\ldots,j+3$. Denote its three highest points by
\[
a_j\quad\text{at height }j+3,\qquad
b_j=\sigma(a_j),\qquad c_j=\sigma(b_j).
\]
Let $\sigma$ move down each branch, fixing $r$, and let the markers
collapse their respective heights to $r$. Impose the comparisons
\[
b_{j+1}\leqslant a_j\qquad(0\leqslant j<n-1)
\]
and their successive $\sigma$-images, reversing each inequality at
every application. All comparisons relate equal heights. At each
height the orientations of successive adjacent-branch comparisons
alternate; they form a fence, with no strict order path of length two.
Thus they define a poset $P$, $\sigma$ reverses its order, and every
marker preserves it. In particular, every $a_j$ and every $c_j$ is
maximal, and the only strict upper neighbours of $b_j$ are
$a_{j-1}$ and $c_{j+1}$, when these indices exist.

To check that $\operatorname{Up}(P)$ belongs to $\mathcal C_m$, pad
every branch above $a_j$ to height $m$, making the added points
order-isolated. The resulting space $Q$ is an equivariant monotone
image of a disjoint union of $n$ copies of the discrete $X_m$.
Inverse image therefore embeds $\operatorname{Up}(Q)$ into $A_m^n$.
The induced subposet $P$ is invariant under every dual map, so
restriction is a surjective homomorphism
$\operatorname{Up}(Q)\to\operatorname{Up}(P)$. This proves the
required membership without changing any height label.

Let $C=\{c_0,\ldots,c_{n-1}\}$, an upset, and use the singleton seed
$\{C\}$ in $\operatorname{Up}(P)$. Points below $c_j$ along branch
$j$ never enter a profile: they are absent initially, and their odd
images are still lower points or the root. Every $b_j$ is in the free
region from the start, since $\sigma(b_j)=c_j\in C$. The point $a_j$
is in the free region precisely when $b_j$ has entered: all its other
odd predecessors lie below $c_j$. Markers add no further odd images,
since a reset reaches the root, which is also reached by an odd power
of $\sigma$.

The exact maximal profiles of the successive deductions are therefore
\begin{equation}\label{EqCornishSharpProfiles}
U_k=C\cup\{b_j\mid 0\leqslant j<n, j<k\}
       \cup\{a_j\mid 0\leqslant j<n, j<k-1\}
\qquad(k\geqslant0).
\end{equation}
Indeed, their free regions are
\[
L(U_k)=\{b_0,\ldots,b_{n-1}\}
        \cup\{a_j\mid 0\leqslant j<n, j<k\}.
\]
All the indicated $a_j$ can be added because they are maximal. The
upper-cone test admits $b_j$ when $a_{j-1}$ is available, the other
possible upper neighbour $c_{j+1}$ being present from the start.
This proves \eqref{EqCornishSharpProfiles} inductively using
\eqref{EqCornishAnchorProfiles}. At every positive step the deductive
set consists of all upsets contained in $U_k$.

The last point $a_{n-1}$ first belongs to $U_{n+1}$, and not to $U_n$.
Since $\{a_{n-1}\}$ is an upset, it witnesses a strict $(n+1)$st
deductive step. As $n+1=m-1$, the upper bound is sharp.
\end{proof}

\subsection{An infinite deductive rank}\label{SubsecInfiniteDeduction}

The remaining building block is a familiar variety.

\begin{theorem}\label{ThmMonoids}
The variety $\mathcal M$ of commutative monoids, with distinguished identity $0$, has rank pair $(1,\infty)$.
\end{theorem}

\begin{proof}
For a non-empty subset $I$ of a commutative monoid,
\[
\mathcal I_0(I)=I+\langle I\rangle,
\qquad
\mathcal D_0(I)=\{x\mid x+y\in I\text{ for some }y\in\langle I\rangle\},
\]
where $\langle I\rangle$ is the submonoid generated by $I$. These formulae follow by collecting the parameters in a monoid polynomial. The first output is the subsemigroup generated by $I$, so induction is idempotent; it is not the identity operator in the additive monoid $\mathbb N_0$.

For deduction, use the free commutative monoid on $e_1,e_2,\ldots$, write $e_0=0$, and take
\[
I=\{e_j+e_{j+1}\mid j\geqslant0\}.
\]
For every $k\geqslant1$ the exact iterates are
\begin{equation}\label{EqMonoidInfiniteChain}
\mathcal D_0^k(I)=I\cup\{0,e_1,\ldots,e_{k+1}\}.
\end{equation}
Indeed, the first step adjoins $0$ and $e_2$, using the anchors $e_1$ and $e_1+e_2$. More generally, a residual $x$ with $x+y$ in the current set is a summand of either a single generator or one of the adjacent sums $e_j+e_{j+1}$. A non-zero summand $y$ belonging to the generated submonoid can remove one endpoint of such a sum only if that endpoint is already present; if $y$ is the whole sum, the residual is $0$. Thus the only new generator at the next step is the next endpoint $e_{k+2}$, and it is obtained from $e_{k+1}+e_{k+2}$. This proves~\eqref{EqMonoidInfiniteChain} by induction. The chain is strictly increasing at every step, so deductive rank is infinite.
\end{proof}

\subsection{Completion of the general rank spectrum theorem}\label{SubsecCompleteRealization}

\begin{proof}[Proof of Theorem~\ref{ThmCompleteSpectrum}]
Theorem~\ref{ThmRankZero} settles all pairs with a zero coordinate.
Let $i,d\geqslant1$ be finite. If $i=d$, use $\mathcal P_i$ from
Theorem~\ref{ThmHMSharpRanks}. If $i>d$, use the finite-generated
variety $\mathcal F_{i-1}$ of pair $(i,1)$ from
Theorem~\ref{ThmAnnihilatingSpectrum}, and form
\[
\mathcal F_{i-1}\boxtimes\mathcal P_d.
\]
Theorem~\ref{ThmPermutableProductRanks} gives its pair as $(i,d)$.
If $i<d$, use instead
\[
\mathcal C_{d+1}\boxtimes\mathcal P_i,
\]
where Theorem~\ref{ThmCornishRanks} gives
$\operatorname{rk}(\mathcal C_{d+1})=(1,d)$.
The same product theorem gives pair $(i,d)$.
All these factors have finite signatures and are generated by finite
algebras. Equation~\eqref{EqIndependentProductGenerator} preserves
these properties, proving the finite-generator assertion.

For $r\geqslant1$ finite, Theorems~\ref{ThmInfiniteInduction},
\ref{ThmMonoids} and~\ref{ThmPermutableProductRanks} give
\[
\operatorname{rk}(\mathcal S_\infty\boxtimes\mathcal P_r)=(\infty,r),
\qquad
\operatorname{rk}(\mathcal M\boxtimes\mathcal P_r)=(r,\infty).
\]
Finally, $\mathcal S_\infty\boxtimes\mathcal M$ contains copies of
both expanded component varieties. Its inductive rank is therefore
infinite because that of $\mathcal S_\infty$ is infinite, and its
deductive rank is infinite because that of $\mathcal M$ is infinite.
This realizes $(\infty,\infty)$ and completes the spectrum.
\end{proof}

The witnesses are collected in Table~\ref{TabCompleteSpectrum}. The
role of the permutable factor is important: the product theorem is
not an assertion that arbitrary independent products have ranks
given by coordinatewise maxima.

\begin{example}\label{ExProductNotMaxima}
Let $\mathcal E$ be the variety of sets with a distinguished constant $0$.
It has pair $(0,1)$ by Theorem~\ref{ThmRankZero}, whereas
$\mathcal E\boxtimes\mathcal E$ has pair $(1,1)$.

Indeed, for a non-empty subset $I$ of a product $A\times B$, write
$P=\pi_A I$, $Q=\pi_B I$.
Formula~\eqref{EqProductStarRelations} gives
\[
R_I=(\Delta_A\cup(P\times\{0_A\}))
       \times(\Delta_B\cup(Q\times\{0_B\})).
\]
Both factors are transitive. Induction leaves the two projections
$P,Q$ unchanged, while deduction replaces them by
$P\cup\{0_A\},Q\cup\{0_B\}$. The displayed relation is unchanged
in either case, so transitivity makes both one-step operators
idempotent.

Induction is not the identity. In $\{0,1\}^2$, put
$I=\{(1,0),(0,1)\}$ and use the polynomial
$q(z)=p(z,(0,1))$, where $p$ is the selector. Then
\[
q(0,0)=(0,1)\in I,\qquad
q(1,0)=(1,1)\notin I.
\]
Thus inductive rank is one; deductive rank is one by non-triviality.
In particular, the product's inductive rank exceeds the maximum of
the two factor ranks.
\end{example}

\begin{table}[H]
\centering
\small
\setlength{\tabcolsep}{5pt}
\caption{Witnesses for the general rank spectrum theorem. All parameters are
positive finite integers unless indicated otherwise.}
\label{TabCompleteSpectrum}
\begin{tabular}{@{}lll@{}}
\toprule
Pair & Witness & Results\\
\midrule
$(0,0)$ & Trivial variety & \ref{ThmRankZero}\\[3pt]
$(0,1)$ & Pointed sets & \ref{ThmRankZero}\\[3pt]
$(r,r)$ & $\mathcal P_r$ & \ref{ThmHMSharpRanks}\\[3pt]
$(i,d)$, $i>d$ & $\mathcal F_{i-1}\boxtimes\mathcal P_d$
 & \ref{ThmAnnihilatingSpectrum}, \ref{ThmPermutableProductRanks}\\[3pt]
$(i,d)$, $i<d$ & $\mathcal C_{d+1}\boxtimes\mathcal P_i$
 & \ref{ThmCornishRanks}, \ref{ThmPermutableProductRanks}\\[3pt]
$(\infty,r)$ & $\mathcal S_\infty\boxtimes\mathcal P_r$
 & \ref{ThmInfiniteInduction}, \ref{ThmPermutableProductRanks}\\[3pt]
$(r,\infty)$ & $\mathcal M\boxtimes\mathcal P_r$
 & \ref{ThmMonoids}, \ref{ThmPermutableProductRanks}\\[3pt]
$(\infty,\infty)$ & $\mathcal S_\infty\boxtimes\mathcal M$
 & \ref{ThmCompleteSpectrum}\\
\bottomrule
\end{tabular}
\end{table}

\section{Special rank spectrum theorems}\label{SecSpecialSpectra}

The general rank spectrum theorem describes the pairs that occur without
additional assumptions on the variety. It is also natural to fix a class
$\mathscr C$ of varieties with a distinguished constant and to determine
its \emph{rank-pair spectrum}
\[
\operatorname{Spec}(\mathscr C)
=\{\operatorname{rk}(V)\mid V\in\mathscr C\}.
\]
The following special rank spectrum theorems illustrate the restrictions
that arise in several classes and in a prescribed family of action
varieties. We state their spectra for non-trivial varieties; adjoining
the trivial variety to a class adds $(0,0)$ to its spectrum.

\subsection{Subtractive varieties}

\begin{theorem}\label{ThmSpecialSubtractiveSpectrum}
The class of non-trivial $0$-subtractive varieties has rank-pair spectrum
\[
\{(1,1),(2,2)\}.
\]
Its subclasses consisting of Mal'tsev varieties, or of varieties having
a J\'onsson--Tarski term for $0$, both have spectrum $\{(1,1)\}$.
\end{theorem}

\begin{proof}
The relation $R_I$ is symmetric in a $0$-subtractive variety, as shown
after Theorem~\ref{ThmC}. Consequently the two one-step operators agree,
and hence so do their ranks. Theorem~\ref{ThmK} bounds these ranks by
$2$. In a non-trivial algebra, deduction of $\{a\}$ for $a\ne0$ adds
$0$, so neither rank is zero. Non-trivial groups realize $(1,1)$ by
Theorem~\ref{PropStandardRanks}.

For $(2,2)$, take $B=\{0,a,b\}$ with distinct elements and define
\[
s(x,y)=\begin{cases}x,&y=0,\\0,&y\ne0.\end{cases}
\]
This is a subtraction operation. The reflexive symmetric relation
\[
T=\Delta_B\cup(B\times\{0\})\cup(\{0\}\times B)
\]
is compatible with $s$: when the second input pair is $(0,0)$ the
output is the first input pair, and otherwise at least one output
coordinate is zero. Starting from $(a,0)$ and diagonal pairs, we obtain
\[
s((a,a),(a,0))=(0,a),\qquad
s((b,b),(0,a))=(b,0),\qquad
s((b,b),(b,0))=(0,b).
\]
Thus $R_{\{a\}}=T=R_{\{0,a\}}$, and both one-step operators give
the strict chain $\{a\}\mapsto\{0,a\}\mapsto B$.
The variety generated by $(B;s,0)$ therefore has pair $(2,2)$.

Finally, Corollary~\ref{CorB} and
Theorem~\ref{ThmStrongUnitalRank} give the upper bound one in the two
subclasses, and non-trivial groups belong to both of them.
\end{proof}

\subsection{Monotone semilattice expansions}

Recall that $\mathcal Q_N$ consists of meet-semilattices with greatest
element $0$ and $N$ contractive meet-endomorphisms whose pairwise
composites are the same constant map, as in
\eqref{EqAnnihilatingSpectrumAxioms}. For $N=0$ there are no additional
unary operations.

\begin{theorem}\label{ThmSpecialMeetSpectrum}
Let $\mathscr S$ be the class of non-trivial varieties of expansions
of meet-semilattices with greatest element $0$ by order-preserving
operations. Then
\[
\operatorname{Spec}(\mathscr S)
=\{(i,1)\mid i\in\overline{\mathbb N}_{>0}\}.
\]
For each fixed $N\geqslant0$, the non-trivial subvarieties of
$\mathcal Q_N$ have spectrum
\[
\{(i,1)\mid 1\leqslant i\leqslant N+1\}.
\]
\end{theorem}

\begin{proof}
Lemma~\ref{LemMeetDeductiveClosure} identifies deduction with upward
closure, so its rank is at most one and, by non-triviality, exactly
one. Inductive rank is also positive: in the square of a non-trivial
algebra, choose $a\ne0$. Induction of $\{(a,0),(0,a)\}$ contains
their missing meet $(a,a)$, using the polynomial
$q(x)=x\wedge(0,a)$. Theorem~\ref{ThmAnnihilatingSpectrum} supplies
every finite positive inductive rank, and
Theorem~\ref{ThmInfiniteInduction} supplies infinite inductive rank.
This proves the first assertion.

For the second, passing to a subvariety cannot increase either rank,
so Theorem~\ref{ThmAnnihilatingSpectrum} gives the asserted upper
bound. To realize $2\leqslant i\leqslant N+1$, expand $L_{i-1}$ by
interpreting each additional operation $f_i,\ldots,f_N$ as $f_1$.
The identities defining $\mathcal Q_N$ still hold: the new operations
are contractive meet-endomorphisms and every composite of two of the
operations is constantly the same least element $c$. Duplicating
operations does not change the polynomial operations in any member
of the generated variety, so the rank pair remains $(i,1)$.
For $(1,1)$, take a two-element meet-semilattice with greatest element
$0$ and, when $N>0$, interpret every $f_j$ as the constant map to its
least element. Every non-constant polynomial is then a meet of inputs
and parameters. Induction is closure under non-empty finite meets,
and deduction is upward closure. Both have rank one in its generated
variety.
\end{proof}

\subsection{Cornish algebras}

\begin{theorem}\label{ThmSpecialCornishSpectrum}
The class of non-trivial varieties of Cornish algebras, with the lattice
bottom distinguished, has rank-pair spectrum
\[
\{(1,d)\mid d\in\overline{\mathbb N}_{>0}\}.
\]
Every finite pair in this spectrum is realized by a variety generated
by a finite algebra in a finite signature.
For each fixed $m\geqslant2$, the non-trivial subvarieties of
$\mathcal C_m$ have spectrum
\[
\{(1,d)\mid 1\leqslant d\leqslant m-1\}.
\]
\end{theorem}

\begin{proof}
Lemma~\ref{LemCornishInduction} gives inductive rank at most one.
It is not zero in a non-trivial variety: in the square of a non-trivial
algebra, the seed $\{(1,0),(0,1)\}$ acquires $(1,1)$ by the polynomial
$q(x)=x\vee(0,1)$. Deductive rank is positive by
Theorem~\ref{ThmRankZero}. Theorem~\ref{ThmCornishRanks} realizes all
the finite pairs with the stated finiteness properties.

To realize $(1,\infty)$, use the common countable signature
\[
(\wedge,\vee,0,1,f,e_1,e_2,\ldots).
\]
Expand each finite algebra $A_m$ of
Subsection~\ref{SubsecCornishRanks} to this signature by interpreting
$e_j$ as the identity whenever $j>m$. Denote the expansion by
$\widehat A_m$, and let
$\mathcal C_\infty=\operatorname{HSP}\{\widehat A_m\mid m\geqslant2\}$.
This is a variety of Cornish algebras. The variety generated by
$\widehat A_m$ is the expansion of $\mathcal C_m$ by identity
operations; in particular it has the same polynomial operations and
the same ranks, namely $(1,m-1)$. It is a subvariety of
$\mathcal C_\infty$, so the latter has no finite deductive-rank bound.
Its inductive rank is one by the first paragraph. This supplies the
remaining pair.

Finally, ranks cannot increase on passage to a subvariety, so
Theorem~\ref{ThmCornishRanks} bounds the deductive ranks in
$\mathcal C_m$ by $m-1$. For $2\leqslant k\leqslant m$, the restriction
map
\[
\mathcal P(X_m)\longrightarrow\mathcal P(X_k),\qquad
U\longmapsto U\cap X_k,
\]
is a surjective homomorphism from $A_m$ onto the expansion of $A_k$
in which $e_j$ is the identity for $k<j\leqslant m$. Indeed, $X_k$
is invariant under $\sigma$ and every marker, and the restrictions of
the first $k$ markers are precisely those defining $A_k$. The added
identity operations do not change polynomial operations. The variety
generated by this quotient is therefore a subvariety of $\mathcal C_m$
with pair $(1,k-1)$. These are all the asserted finite pairs.
\end{proof}

\subsection{Finite group actions and Cayley-graph diameters}\label{SubsecGroupActions}

Let $G$ be a finite group, with identity $1$, and let $\mathcal A_G$
be the variety of left $G$-sets with an unconstrained named constant
$c$. Its unary operations $\lambda_g$, for $g\in G$, satisfy
\[
\lambda_1(x)=x,\qquad
\lambda_g(\lambda_h(x))=\lambda_{gh}(x).
\]
No identity $\lambda_g(c)=c$ is imposed. The ranks of $\mathcal A_G$
are taken at $c$. These algebras carry an action of $G$, not a group
multiplication on their underlying sets, so the Mal'tsev rank bound
does not apply.

For $S\subseteq G$, put
\[
S^{\leqslant q}=\bigcup_{j=0}^q S^j,\qquad S^0=\{1\}.
\]
If $H=\langle S\rangle$, the \emph{positive diameter}
\[
d^+(H,S)=\min\{q\geqslant0\mid S^{\leqslant q}=H\}
\]
is the diameter of the directed Cayley graph with edges $x\to xs$,
$s\in S$. If $1\in S=S^{-1}$, write
$d(H,S)=\min\{q\geqslant0\mid S^q=H\}$ for its undirected
diameter. These are the generation-length invariants studied in
\cite{KL03,KL09}. We use their maxima over subgroups:
\begin{align*}
\Delta^+_{\rm sub}(G)
 &=\max\{d^+(H,S)\mid H\leqslant G,\ \langle S\rangle=H\},\\
\Delta_{\rm sub}(G)
 &=\max\{d(H,S)\mid H\leqslant G,\ 1\in S=S^{-1},\
                              \langle S\rangle=H\}.
\end{align*}

\begin{theorem}\label{ThmGroupInductiveDiameter}
For every finite group $G$,
\[
r_{\mathcal I}(\mathcal A_G)
=\left\lceil\log_2\bigl(1+\Delta^+_{\rm sub}(G)\bigr)\right\rceil.
\]
\end{theorem}

\begin{proof}
Every non-constant polynomial on a $G$-set is a translate of one
variable. Hence, for a non-empty seed $I\subseteq A$,
\begin{equation}\label{EqGroupActionRelation}
R_I=\Delta_A\cup\{(gi,gc)\mid g\in G,\ i\in I\}.
\end{equation}
Put $S=\{g\in G\mid gc\in I\}$. Then
$\mathcal I_c(I)=I\cup SI$. If
$I_k=\mathcal I_c^k(I)$ and $S_k=\{g\mid gc\in I_k\}$, we obtain
\[
I_{k+1}=I_k\cup S_kI_k,\qquad
S_{k+1}=S_k\cup S_kS_k.
\]
Indeed, $uc=gi$ with $g\in S_k$ and $i\in I_k$ is equivalent
to $g\in S_k$ and $g^{-1}u\in S_k$. Thus no assumption on the
stabilizer of $c$ is needed. Induction gives
\begin{equation}\label{EqGroupInductionBalls}
S_k=\bigcup_{j=1}^{2^k}S^j,\qquad
I_k=S^{\leqslant 2^k-1}I.
\end{equation}
If $S$ is empty, $I$ is fixed. Otherwise let $H=\langle S\rangle$.
Once $2^k-1\geqslant d^+(H,S)$, the second formula gives
$I_k=HI$, which is fixed. This proves the upper bound.

For sharpness, choose $H,S$ attaining
$D=\Delta^+_{\rm sub}(G)>0$. Take two regular orbits
$A=G\times\{0,1\}$, with $g(h,j)=(gh,j)$, and put
\[
c=(1,0),\qquad I=(S\times\{0\})\cup\{(1,1)\}.
\]
The second-orbit section of $I_k$ is
$S^{\leqslant 2^k-1}\times\{1\}$, while its eventual section is
$H\times\{1\}$. Therefore $I_k$ is not fixed when $2^k-1<D$.
This gives the lower bound. If $D=0$, the upper bound suffices.
\end{proof}

The second orbit in this proof cannot in general be omitted.
For example, the regular action of $C_2^2$ alone has inductive
rank one, whereas $\mathcal A_{C_2^2}$ has inductive rank two.
Indeed, in the regular orbit induction is $F\mapsto F\cup(F+F)$,
which is idempotent on the four-element group. With a basis $a,b$
and a marker in a second regular orbit, the first marker section is
$\{0,a,b\}$, and $a+b$ appears only at the second step.

Deduction also has a description by word balls, but a symmetrization
is needed first. The following lemma deals in particular with
non-free actions.

\begin{lemma}\label{LemGroupDeductionBalls}
To compute $r_{\mathcal D}(\mathcal A_G)$, it suffices to use finite
unions of regular $G$-orbits. For a regular distinguished orbit,
identified with $G$ and with $c=1$, let $F$ be the seed on that
orbit and $J$ the seed outside it. Put
\[
E=\{g\in G\mid gJ\cap J\ne\varnothing\}\cup\{1\},\qquad
B=F\cup FF^{-1}\cup E,\qquad X=BB^{-1}.
\]
The distinguished-orbit sections of the deductive iterates are
\begin{equation}\label{EqGroupDeductionBalls}
F_1=B,\qquad F_n=X^{2^{n-2}}\quad(n\geqslant2).
\end{equation}
Their eventual value is $H=\langle B\rangle$. If $B\ne H$,
the exact stabilization stage is
\begin{equation}\label{EqGroupDeductionTime}
2+\left\lceil\log_2 d(H,BB^{-1})\right\rceil.
\end{equation}
If $B=H$, stabilization occurs by stage one.
\end{lemma}

\begin{proof}
Formula~\eqref{EqGroupActionRelation} shows that deduction changes
only the orbit of $c$. The external collision set is symmetric.
Adjoining $1$ to it leaves the recurrence unchanged: if $F$ is
non-empty, $1\in FF^{-1}$; otherwise the non-empty seed lies outside
the distinguished orbit, and its collision set already contains $1$.
For a regular distinguished orbit, the recurrence is
\[
F_{n+1}=F_n\cup F_nF_n^{-1}\cup E.
\]
Since $\{1\}\cup E\subseteq F_1$, it becomes
$F_{n+1}=F_nF_n^{-1}$ for $n\geqslant1$. The set
$X=F_1F_1^{-1}$ is symmetric and contains $1$, proving
\eqref{EqGroupDeductionBalls}. A set containing $1$ and closed under
$xy^{-1}$ is a subgroup, so the eventual value is $H$ and a
positive stage is fixed exactly when it equals $H$. This proves
\eqref{EqGroupDeductionTime}.

Now let the distinguished orbit be $G/K$, where $K$ need not be
normal. Lift its stage sections to right-$K$-saturated subsets
$A_n\subseteq G$. With the external collision set $E$, formula
\eqref{EqGroupActionRelation} gives
\[
A_1=A_0\cup A_0A_0^{-1}K\cup EK,\qquad
A_{n+1}=A_nA_n^{-1}K\quad(n\geqslant1).
\]
Here $K\subseteq A_1$, and the eventual value is
$H=\langle A_0\cup E\cup K\rangle$: a fixed stage containing
$K$ is a subgroup, and every stage is contained in $H$.

Every symmetric subset $T\subseteq G$ containing $1$ is an external
collision set on a finite union of regular orbits. In separate
orbits put seed pairs $\{1,t\}$, for $t\in T$; their collision
sets are $\{1,t,t^{-1}\}$ and have union $T$. We may therefore
simulate the original process on a regular distinguished orbit,
with initial section $A_0$ and external collision set $E\cup K$.
Its first section is
\[
B_1=A_0\cup A_0A_0^{-1}\cup E\cup K\subseteq A_1.
\]
Inductively $B_n\subseteq A_n$, and both eventual sections are
$H$. Whenever the original process has not stabilized at a
positive stage, the simulation has not stabilized there either.
Since deduction has rank at least one, even for the trivial group,
this proves the reduction.
\end{proof}

\begin{theorem}\label{ThmGroupDeductiveDiameter}
For every non-trivial finite group $G$,
\[
1+\left\lceil\log_2\Delta_{\rm sub}(G)\right\rceil
\ \leqslant r_{\mathcal D}(\mathcal A_G)\leqslant\
2+\left\lceil\log_2\Delta_{\rm sub}(G)\right\rceil.
\]
More precisely, its deductive rank is the maximum of $1$ and the
values in \eqref{EqGroupDeductionTime}, as
\[
B=F\cup FF^{-1}\cup E,\qquad
F\subseteq G,\quad 1\in E=E^{-1},
\]
ranges over the sets which are not subgroups.
\end{theorem}

\begin{proof}
The lemma gives the upper bound and the exact maximum formula,
since every indicated pair $F,E$ can be realized. For the lower
bound, choose $H,S$ attaining $\Delta_{\rm sub}(G)$, with
$1\in S=S^{-1}$. Take an empty initial distinguished-orbit section
and realize $S$ as an external collision set. Then
$F_n=S^{2^{n-1}}$ for $n\geqslant1$, so stabilization first occurs
at $1+\lceil\log_2 d(H,S)\rceil$.
\end{proof}

\subsection{Finite abelian actions}\label{SubsecAbelianActions}

For finite abelian groups both ranks can be determined exactly.
Write the invariant-factor decomposition of a non-trivial group as
\[
G=C_{m_1}\oplus\cdots\oplus C_{m_t},\qquad
1<m_1\mid\cdots\mid m_t,
\]
and put
\begin{equation}\label{EqAbelianDiameterParameters}
P=\sum_{i=1}^t(m_i-1),\qquad
Q=\sum_{i=1}^t\left\lfloor\frac{m_i}{2}\right\rfloor,\qquad
e=m_t.
\end{equation}
Thus $e$ is the exponent of $G$. Klopsch and Lev prove that the
maximum positive diameter of $G$ is $P$
\cite[Theorem~2.1]{KL09}, and its maximum undirected diameter is
$Q$ \cite[Theorem~2.1]{KL03}. These are also the corresponding
maxima over subgroups. Indeed, if the invariant factors of a
subgroup $H\leqslant G$ are padded by initial ones to give $n_1,\ldots,n_t$, then
$n_i\mid m_i$ for every $i$, so neither sum increases. To see the
divisibility, compare, for each prime $p$ and $k\geqslant0$, the
dimensions of $(p^kH_p)[p]\leqslant(p^kG_p)[p]$ over $\mathbb F_p$,
where $H_p,G_p$ are the $p$-primary components and
$B[p]=\{b\in B\mid pb=0\}$. These dimensions count the invariant
factors whose $p$-exponents exceed $k$, giving the componentwise
inequalities between the ordered exponents.

\begin{theorem}\label{ThmAbelianActionRanks}
For every non-trivial finite abelian group $G$, with parameters
\eqref{EqAbelianDiameterParameters},
\begin{equation}\label{EqAbelianActionRanks}
\operatorname{rk}(\mathcal A_G)
=
\left(
\left\lceil\log_2(P+1)\right\rceil,\
1+\left\lceil\log_2\max\{e-1,Q\}\right\rceil
\right).
\end{equation}
For the trivial group, $\operatorname{rk}(\mathcal A_G)=(0,1)$.
\end{theorem}

\begin{proof}
The inductive formula follows from
Theorem~\ref{ThmGroupInductiveDiameter}. For deduction use additive
notation and the regular-orbit reduction of
Lemma~\ref{LemGroupDeductionBalls}. Thus
\[
W=F\cup(F-F)\cup E,\qquad X=W-W,\qquad H=\langle W\rangle,
\]
where $0\in E=-E$, and $F_n=2^{n-2}X$ for $n\geqslant2$.

We first prove the upper bound. If $F$ is empty, then
$F_n=2^{n-1}E$ for $n\geqslant1$, and the symmetric diameter
bound $Q$ gives stabilization by stage
$1+\lceil\log_2 Q\rceil$. Suppose $F$ is non-empty, choose
$a\in F$, and put $U=(F-F)\cup E$. The symmetric set
$T=U\cup\{a,-a\}$ generates $H$, since
$f=a+(f-a)$ for $f\in F$. Moreover,
\begin{equation}\label{EqAbelianPackingInclusions}
U+U\subseteq X,\qquad a+U\subseteq X,\qquad -a+U\subseteq X.
\end{equation}
By the undirected diameter bound, every $h\in H$ has a word in
$T$ of length at most $Q$. Commuting its letters, cancelling
opposite occurrences of $a$, and reducing the coefficient of $a$
modulo its order gives
\[
h=ka+u_1+\cdots+u_\ell,\qquad
|k|+\ell\leqslant Q,\qquad
|k|\leqslant\lfloor e/2\rfloor,\qquad u_j\in U.
\]
Reduction of the coefficient does not increase its absolute value.
Pair each remaining $a$-letter with a $U$-letter when available,
and pair the remaining $U$-letters. Since $0\in U$, the inclusions
\eqref{EqAbelianPackingInclusions} express $h$ as a sum of at most
\[
\max\left\{|k|,\left\lceil\frac{|k|+\ell}{2}\right\rceil\right\}
\leqslant
\max\left\{\left\lfloor\frac e2\right\rfloor,
           \left\lceil\frac Q2\right\rceil\right\}
\]
members of $X$.

Set $M=\max\{e-1,Q\}$ and $p=2^{\lceil\log_2M\rceil}$.
If $M\geqslant2$, then $p$ is even, and $p\geqslant e-1,Q$
implies
\[
p/2\geqslant
\max\{\lfloor e/2\rfloor,\lceil Q/2\rceil\}.
\]
Consequently $F_n=H$ by stage
$n=1+\lceil\log_2M\rceil$. The remaining case $M=1$ is
$G=C_2$. Its first section $W$ is always a subgroup, so deduction
is idempotent.

For the first lower bound, choose an element $a$ of order $e$
in a regular distinguished orbit and take the singleton seed
$\{a\}$. Its sections are
\[
F_1=\{0,a\},\qquad
F_n=\{ja\mid -2^{n-2}\leqslant j\leqslant2^{n-2}\}
\quad(n\geqslant2).
\]
For $e\geqslant3$, these first fill $\langle a\rangle$ at stage
$1+\lceil\log_2(e-1)\rceil$; for $e=2$, stage one suffices.
For the other lower bound, let $b_1,\ldots,b_t$ be a standard
invariant-factor basis and realize
$E=\{0,\pm b_1,\ldots,\pm b_t\}$ as an external collision set,
starting the distinguished orbit empty. Its symmetric diameter
is $Q$, and the sections $F_n=2^{n-1}E$ first fill $G$ at
stage $1+\lceil\log_2Q\rceil$. The maximum of these lower
bounds is the claimed value.

For the trivial group, induction fixes every set and deduction
adjoins $c$, giving $(0,1)$.
\end{proof}

The two parameters in the deductive formula have different roles.
The diameter $Q$ measures generation using a symmetric set, while
$e-1$ accounts for the initial asymmetry of a singleton in a cyclic
direction. Neither can in general be omitted. For example, $C_4$
and $C_2^2$ both have maximum undirected diameter two, but
\[
\operatorname{rk}(\mathcal A_{C_4})=(2,3),\qquad
\operatorname{rk}(\mathcal A_{C_2^2})=(2,2).
\]
For an elementary abelian group $(C_p)^t$, with $p$ prime and
$t\geqslant1$, formula~\eqref{EqAbelianActionRanks} becomes
\[
\left(
\left\lceil\log_2\bigl(1+t(p-1)\bigr)\right\rceil,\
1+\left\lceil\log_2\max\{p-1,t\lfloor p/2\rfloor\}\right\rceil
\right).
\]
In particular, $\mathcal A_{C_3^2}$ has pair $(3,2)$ and
$\mathcal A_{C_3^4}$ has pair $(4,3)$.

\begin{theorem}\label{ThmSpecialAbelianActionSpectrum}
The family $\mathscr C_{\rm ab}
=\{\mathcal A_G\mid G\text{ is a finite abelian group}\}$ has spectrum
\begin{equation}\label{EqAbelianActionSpectrum}
\operatorname{Spec}(\mathscr C_{\rm ab})
=\{(0,1)\}\cup\{(r,r)\mid r\geqslant1\}
 \cup\{(r,r+1),(r+1,r)\mid r\geqslant2\},
\end{equation}
where all parameters are finite integers.
\end{theorem}

\begin{proof}
For non-trivial $G$, put $M=\max\{e-1,Q\}$. The parameters above satisfy $M\leqslant P
\leqslant2M$. Formula~\eqref{EqAbelianActionRanks} therefore implies
that the rank coordinates differ by at most one. Either coordinate
equals one only when $G=C_2$, whose pair is $(1,1)$.
The trivial group gives $(0,1)$, and the following groups realize
all the other pairs:
\[
\begin{array}{c|c|c}
\text{pair}&G&\text{range}\\ \hline
(r,r)&C_{\,2^{r-1}+1}&r\geqslant1\\
(r,r+1)&C_{\,2^r}&r\geqslant2\\
(r+1,r)&(C_3)^{\,2^{r-1}}&r\geqslant2.
\end{array}
\]
\end{proof}

\subsection{Finite cyclic actions}\label{SubsecCyclicSpectrum}

For $m\geqslant1$, let $\mathcal V_m$ have signature $(f,0)$ and
defining identity $f^m(x)=x$, where $f$ is unary. This is
term-equivalent to $\mathcal A_{C_m}$: the group elements act
by the powers of $f$, and the named constant need not be fixed.
Thus the cyclic spectrum is a consequence of the preceding results.

\begin{corollary}\label{ThmSpecialCyclicSpectrum}
For $m\geqslant2$,
\[
\operatorname{rk}(\mathcal V_m)
=\bigl(\lceil\log_2m\rceil,\,
       1+\lceil\log_2(m-1)\rceil\bigr).
\]
Consequently the family $\mathscr C_{\rm cyc}
=\{\mathcal V_m\mid m\geqslant2\}$ has spectrum
\[
\operatorname{Spec}(\mathscr C_{\rm cyc})
=\{(r,r)\mid r\geqslant1\}
 \cup\{(r,r+1)\mid r\geqslant2\},
\]
where $r$ ranges over integers. Including $\mathcal V_1$ adds
the pair $(0,1)$.
\end{corollary}

\begin{proof}
In Theorem~\ref{ThmAbelianActionRanks}, take
$P=e-1=m-1$ and $Q=\lfloor m/2\rfloor$. Put
$r=\lceil\log_2m\rceil$. The deductive rank is $r$ when
$m=2^{r-1}+1$ and is $r+1$ for
$2^{r-1}+1<m\leqslant2^r$. The choices
$m=2^{r-1}+1$ and $m=2^r$ give the displayed spectrum.
\end{proof}

Although both ranks are unbounded as the group varies, every
finite-group action variety has finite ranks. The spectra above
therefore contain no pair with an infinite coordinate.

\subsection{Ranks at arbitrary elements}\label{SubsecUniversalRanks}

The constructions of Section~\ref{SecClosures} use a chosen element
of an algebra, whether or not it is named by a constant. Thus, for an
arbitrary variety $V$, we may define its \emph{universal inductive rank}
$r_{\mathcal I}^{\mathrm{univ}}(V)$ to be the least non-negative integer $n$
such that
\[
\mathcal I_a^n(I)=\mathcal I_a^{n+1}(I)
\]
for every $A\in V$, every $a\in A$, and every non-empty subset
$I\subseteq A$. Its \emph{universal deductive rank}
$r_{\mathcal D}^{\mathrm{univ}}(V)$ is defined similarly. As before, the
value is $\infty$ when no finite uniform bound exists. Write
\[
\operatorname{urk}(V)
=\bigl(r_{\mathcal I}^{\mathrm{univ}}(V),
       r_{\mathcal D}^{\mathrm{univ}}(V)\bigr).
\]

Let $c$ be a fresh constant symbol. Denote by $V[c]$ the variety in the
enlarged signature axiomatized by the identities defining $V$, with
no additional identities imposed. Its algebras are precisely the
expansions $(A,a)$ with $A\in V$ and $a\in A$, interpreting $c$ as $a$.
The original signature of $V$ may already contain other constants.
For example, if $\mathcal G_G$ denotes the variety of $G$-sets
without a named element, then
$\mathcal A_G=\mathcal G_G[c]$. In particular, the cyclic-action
variety $\mathcal V_m$ is obtained by adjoining $0$ to the unary
variety defined by $f^m(x)=x$.

\begin{theorem}\label{ThmUniversalConstantReduction}
For every variety $V$,
\[
\operatorname{urk}(V)=\operatorname{rk}(V[c]),
\]
where the ranks on the right are taken at $c$.
Consequently, if $\mathscr E$ is the class of all such varieties
$V[c]$, with $c$ distinguished, then
\[
\{\operatorname{urk}(V)\mid V\text{ is a variety}\}
=\operatorname{Spec}(\mathscr E).
\]
\end{theorem}

\begin{proof}
An algebra $A$ and its expansion $(A,a)$ have the same polynomial
operations: every occurrence of $c$ can be replaced by the parameter
$a$. Hence induction and deduction at $a$ in $A$ agree, step by step,
with induction and deduction at $c$ in $(A,a)$. Since $a$ ranges over
all elements of all algebras of $V$, each finite uniform bound on
either side is equivalent to the same bound on the other side.
The ranks therefore agree, including the values $0$ and $\infty$.
The spectrum identity follows.
\end{proof}

In particular, $\operatorname{urk}(\mathcal G_G)
=\operatorname{rk}(\mathcal A_G)$. Thus
Theorems~\ref{ThmGroupInductiveDiameter} and~\ref{ThmGroupDeductiveDiameter}
describe the universal ranks of finite-group action varieties,
and Theorem~\ref{ThmAbelianActionRanks} computes them exactly when
$G$ is abelian. Their spectrum as $G$ ranges over finite abelian
groups is precisely \eqref{EqAbelianActionSpectrum}.

The universal rank spectrum problem is a special rank spectrum
problem for varieties obtained by adjoining an unconstrained constant.
It would be of interest to determine this spectrum. The general rank
spectrum theorem does not settle it: that theorem allows identities
involving the distinguished constant, whereas here the distinguished
constant may name any element. Nor does adjoining such a constant
assert that the resulting category of algebras is pointed.

The special spectra established above are collected in Table~\ref{TabSpecialSpectra}.

\begin{table}[H]
\centering
\small
\setlength{\tabcolsep}{5pt}
\renewcommand{\arraystretch}{1.0}
\caption{Special rank spectra. The parameters $N\geqslant0$ and
$m\geqslant2$ are fixed in the two subvariety rows; $r$ is a finite
integer in the action rows.}
\label{TabSpecialSpectra}
\begin{tabular}{@{}p{0.43\textwidth}p{0.39\textwidth}l@{}}
\toprule
Class of non-trivial varieties & Spectrum & Result\\
\midrule
Mal'tsev varieties
 & $\{(1,1)\}$ & \ref{ThmSpecialSubtractiveSpectrum}\\
$0$-subtractive varieties with a J\'onsson--Tarski term for $0$
 & $\{(1,1)\}$ & \ref{ThmSpecialSubtractiveSpectrum}\\
$0$-subtractive varieties
 & $\{(1,1),(2,2)\}$ & \ref{ThmSpecialSubtractiveSpectrum}\\
Monotone meet-semilattice expansions, greatest element $0$
 & $\{(i,1)\mid i\in\overline{\mathbb N}_{>0}\}$ & \ref{ThmSpecialMeetSpectrum}\\
Subvarieties of $\mathcal Q_N$
 & $\{(i,1)\mid 1\leqslant i\leqslant N+1\}$ & \ref{ThmSpecialMeetSpectrum}\\
Cornish varieties, distinguished bottom
 & $\{(1,d)\mid d\in\overline{\mathbb N}_{>0}\}$ & \ref{ThmSpecialCornishSpectrum}\\
Subvarieties of $\mathcal C_m$
 & $\{(1,d)\mid 1\leqslant d\leqslant m-1\}$ & \ref{ThmSpecialCornishSpectrum}\\
Finite abelian-action family $\mathscr C_{\rm ab}$
 & $\{(0,1)\}$\newline
   $\cup\{(r,r)\mid r\geqslant1\}$\newline
   $\cup\{(r,r+1)\mid r\geqslant2\}$\newline
   $\cup\{(r+1,r)\mid r\geqslant2\}$
 & \ref{ThmSpecialAbelianActionSpectrum}\\
Finite cyclic-action family $\mathscr C_{\rm cyc}$
 & $\{(r,r)\mid r\geqslant1\}$\newline
   $\cup\{(r,r+1)\mid r\geqslant2\}$
 & \ref{ThmSpecialCyclicSpectrum}\\
\bottomrule
\end{tabular}
\end{table}

\section{Detecting structural properties}\label{SecDetection}

The preceding sections ask how many steps induction and deduction
require. We now ask what properties of these operations can tell us
about the algebras themselves. Some of the following tests concern a
particular subset of an algebra or its square; others must hold
throughout a variety. We distinguish a fixed named constant from a
basepoint allowed to range over all elements. As before, equality of
operators means equality on every non-empty subset of every algebra
under consideration.

\subsection{Subtractivity and agreement of the two processes}
\label{SubsecDetSubtractivity}

The symmetry argument preceding Theorem~\ref{ThmK} shows that
subtractivity makes the one-step operators equal. Its converse gives
the following characterisation.

\begin{theorem}\label{ThmDetSubtractivity}
For a variety $V$ with a distinguished constant $0$, the following
conditions are equivalent:
\begin{enumerate}
\item $V$ is subtractive;
\item $\mathcal I_0=\mathcal D_0$;
\item $0\in\mathcal I_0(\{a\})$ for every $a$ in every algebra of $V$.
\end{enumerate}
\end{theorem}

\begin{proof}
If $s$ is a subtractive term and $a\in J$, applying $s$ to the pairs
$(a,a)$ and $(a,0)$ in $R_J$ gives $(0,a)\in R_J$. Hence $R_J$
contains the converse of each of its generators and is symmetric.
Thus $\mathcal I_0(J)=R_JJ=JR_J=\mathcal D_0(J)$, proving
(1)$\Rightarrow$(2). Since deduction from a non-empty set always
contains $0$, (2) implies (3). For the converse,
work in the free algebra $F_V(x)$. The singleton polynomial description
of (3) supplies a unary polynomial $q$ with $q(x)=0$ and $q(0)=x$.
Its parameters are terms in $x$; replacing the polynomial variable by
$y$ therefore gives a binary term $s(x,y)$ satisfying
$s(x,x)=0$ and $s(x,0)=x$. This is a subtractive term.
\end{proof}

Agreement of the eventual closures is weaker than agreement in one
step. Recall that $V$ is \emph{$n$-subtractive}, for $n\geqslant2$,
if it has binary terms $s_1,\ldots,s_{n-1}$ satisfying
\begin{equation}\label{EqDetNSubtractive}
s_1(x,x)=0,\qquad
s_i(x,0)=s_{i+1}(x,x)\ (1\leqslant i<n-1),\qquad
s_{n-1}(x,0)=x.
\end{equation}
Thus $2$-subtractivity is ordinary subtractivity. We use the convention
of \cite[Definition 5.6]{DetCS13}.

\begin{theorem}\label{ThmDetEventualSubtractivity}
For a variety $V$ with a distinguished constant $0$, the following
conditions are equivalent:
\begin{enumerate}
\item $V$ is $n$-subtractive for some finite $n\geqslant2$;
\item every compatible preorder $P$ on every algebra in $V$ satisfies
      $aP0\Longleftrightarrow0Pa$;
\item the non-empty $0$-inductive sets and the non-empty
      $0$-deductive sets coincide;
\item $\mathcal I_0^\infty=\mathcal D_0^\infty$;
\item $0\in\mathcal I_0^\infty(\{a\})$ for every $a$ in every
      algebra of $V$.
\end{enumerate}
When these conditions hold, the common closure is normal closure.
They are also equivalent to requiring that, in every compatible
partial order on every algebra in $V$, the element $0$ be comparable
only with itself.
\end{theorem}

\begin{proof}
If $aP0$, compatibility and \eqref{EqDetNSubtractive} give a chain
from $0$ to $a$; if $0Pa$, the same chain is reversed. Transitivity
proves (1)$\Rightarrow$(2).

Assume (2), let $J$ be non-empty, and let $P$ be the transitive
closure of $R_J$. If $J$ is inductive, Theorem~\ref{ThmC} gives
$PJ=J$. For $a\in J$ we have $aP0$, hence $0Pa$, so $0\in J$.
Consequently $J=P0=0P$. If $J$ is deductive, then $0\in J$ and
$JP=J$. The generating pairs and (2) again give $J=0P=P0$.
In either case $J=P0=0P$ is the class of $0$ for the congruence
$P\cap P^\circ$. It is therefore normal, and hence both inductive
and deductive by Theorem~\ref{ThmA}. This proves (3).

The fixed sets determine the corresponding least closures, so (3)
and (4) are equivalent. Condition (4) implies (5), since every
non-empty deductive closure contains $0$.

Finally assume (5) in $F_V(x)$, and let $R$ be the semicongruence
generated by $(x,0)$. Theorem~\ref{ThmC} gives a finite path
\[
0=u_0\mathrel R u_1\mathrel R\cdots\mathrel R u_k=x.
\]
In a non-trivial variety $k\geqslant1$. Every edge has the form
$(q_i(x),q_i(0))$ for a unary polynomial with parameters in $F_V(x)$.
Writing these parameters as terms in $x$ gives binary terms
$s_i(x,y)$ with $s_i(x,x)=u_{i-1}$ and $s_i(x,0)=u_i$.
These satisfy \eqref{EqDetNSubtractive} with $n=k+1$.
The trivial variety is subtractive, completing the equivalence.

For the last assertion, (2) plainly implies the partial-order
condition. Conversely, quotient a compatible preorder $P$ by the
congruence $P\cap P^\circ$. The induced relation is a compatible
partial order. If $aP0$ or $0Pa$, the images of $a$ and $0$ are
comparable and hence equal, which gives both relations.
This is the order-theoretic equivalence of
\cite[Theorem 5.9]{DetCS13}.
\end{proof}

\begin{corollary}\label{CorDetSubtractiveRankOne}
If $V$ is $n$-subtractive for some finite $n$ and has inductive
rank at most one, then $V$ is subtractive and both ranks are at
most one. If $V$ is non-trivial, its rank pair is $(1,1)$.
\end{corollary}

\begin{proof}
By Theorem~\ref{ThmDetEventualSubtractivity} and the rank bound,
$0\in\mathcal I_0^\infty(\{a\})=\mathcal I_0(\{a\})$.
Apply Theorem~\ref{ThmDetSubtractivity}. Equality of the operators
gives the same deductive bound, and deduction is nontrivial in
every non-trivial variety.
\end{proof}

\subsection{Permutability and universal agreement}
\label{SubsecDetPermutability}

When the basepoint is arbitrary, the corresponding tests detect
congruence permutability rather than subtractivity.

\begin{theorem}\label{ThmDetUniversalAgreement}
For an arbitrary variety $V$ the following conditions are equivalent:
\begin{enumerate}
\item $V$ is a Mal'tsev variety;
\item $\mathcal I_a=\mathcal D_a$ for every basepoint $a$ in every
      algebra of $V$;
\item $a\in\mathcal I_a(\{b\})$ for every pair $a,b$ in every
      algebra of $V$.
\end{enumerate}
The following weaker conditions are also equivalent:
\begin{enumerate}[label=(\roman*)]
\item $V$ is congruence $n$-permutable for some finite $n\geqslant2$;
\item $\mathcal I_a^\infty=\mathcal D_a^\infty$ for every
      basepoint $a$ in every algebra of $V$;
\item $a\in\mathcal I_a^\infty(\{b\})$ for every pair $a,b$ in
      every algebra of $V$.
\end{enumerate}
Under the latter conditions the common closure is normal closure.
\end{theorem}

\begin{proof}
Theorem~\ref{ThmMalcevRank} gives (1)$\Rightarrow$(2), and
deduction of $\{b\}$ contains $a$, giving (2)$\Rightarrow$(3).
Similarly, Theorem~\ref{ThmPermutableRank}, applied at any chosen
basepoint, gives (i)$\Rightarrow$(ii) and the assertion about normal
closure; (ii)$\Rightarrow$(iii) is immediate.

For both converses, work in $F_V(x,y)$ with basepoint $x$ and seed
$\{y\}$. Let $R$ be the semicongruence generated by $(y,x)$.
Condition (iii) and Theorem~\ref{ThmC} give a finite path
\[
x=u_0\mathrel R u_1\mathrel R\cdots\mathrel R u_k=y.
\]
Every edge is $(q_i(y),q_i(x))$ with parameters which are terms
in $x,y$. Denote by $p_i(x,z,y)$ the term obtained by replacing
the polynomial variable with $z$. Then
\[
p_1(x,y,y)=x,\qquad
p_i(x,x,y)=p_{i+1}(x,y,y),\qquad
p_k(x,x,y)=y.
\]
These are the Hagemann--Mitschke identities
\eqref{EqHMIdentities}, giving $(k+1)$-permutability
\cite{HM73}. Under (3) the path can be taken with $k=1$, and
$p_1$ is a Mal'tsev term. Trivial varieties satisfy both conclusions.
\end{proof}

By Hagemann's order-theoretic characterisation, the equivalent
conditions (i)--(iii) also say that every compatible preorder is
symmetric, or equivalently that every compatible partial order on
every algebra in $V$ is equality; see
\cite[Theorem 5.1]{DetCS13} and \cite{JRV14}. Thus universal
agreement of the eventual closures detects the absence of
nontrivial compatible orders throughout a variety.

\begin{corollary}\label{CorDetPermutableRankOne}
Among varieties which are congruence $n$-permutable for some finite
$n$, universal inductive rank at most one characterises Mal'tsev
varieties. Every non-trivial such variety has universal rank
pair $(1,1)$.
\end{corollary}

\begin{proof}
The eventual singleton membership in
Theorem~\ref{ThmDetUniversalAgreement}(iii) becomes one-step
membership under the rank bound. The first equivalence in that
theorem gives a Mal'tsev term. The converse follows from
Theorem~\ref{ThmMalcevRank}. In a non-trivial variety, choose
$a\ne b$ in an algebra. Both operators at $a$ add $a$ to $\{b\}$,
so neither universal rank is zero.
\end{proof}

\subsection{Detecting distributivity in lattices}\label{DetSubsecLattices}

In this subsection, a lattice has its usual operations, together with
any specified constants, but no additional operations of positive arity.
The first characterisation concerns individual lattices and a fixed
basepoint; the second concerns universal ranks of varieties.

\begin{theorem}\label{DetThmLatticeSingletons}
For a bounded lattice $L$, the following conditions are equivalent:
\begin{enumerate}
\item $L$ is distributive;
\item every singleton $\{a\}$ is $0$-inductive;
\item $\mathcal D_0(\{a\})$ is normal for every $a\in L$.
\end{enumerate}
Here $0$ is the least element of $L$.
\end{theorem}

\begin{proof}
Identifying all the variables of a polynomial gives
\[
\mathcal I_0(\{a\})
=\{p(a)\mid p\text{ is a unary polynomial},\ p(0)=a\}.
\]
If $L$ is distributive, every unary polynomial has the form
$p(z)=p(0)\vee(z\wedge p(1))$; see
\cite[Propositions 8 and 10]{DetCM12}. Thus $p(0)=a$ implies $p(a)=a$,
proving (1)$\Rightarrow$(2).

For the converse, recall the Dedekind--Birkhoff characterisation by
the forbidden sublattices $M_3,N_5$; see
\cite[Theorem 1]{DetLS17}. In a copy of $M_3$, write $o,t$ for the local
bounds and $a,b,c$ for its three atoms. The polynomial
\[
p(z)=a\vee\bigl((z\vee b)\wedge(z\vee c)\bigr)
\]
satisfies $p(0)=a$ and $p(a)=t\ne a$. In a copy of $N_5$, write
$o<a<b<t$ for its four-element chain and $c$ for the remaining element.
Then
\[
p(z)=a\vee\bigl(b\wedge(z\vee c)\bigr)
\]
satisfies $p(0)=a$ and $p(a)=b\ne a$. These calculations use the global
least element $0$, even when $o>0$. Either sublattice therefore
contradicts (2).

In every bounded lattice,
\begin{equation}\label{DetEqLatticeDeductionSingleton}
\mathcal D_0(\{a\})=\mathord\downarrow a.
\end{equation}
Indeed, monotonicity gives the inclusion from left to right; for
$b\leqslant a$, the polynomial $p(z)=z\vee b$ gives the reverse
inclusion. If $L$ is distributive, the map $x\mapsto x\vee a$ is a
lattice homomorphism and has $\mathord\downarrow a$ as the class of
$0$ in its kernel. Named constants are interpreted by their images in
the codomain. Hence (1) implies (3). Finally, if
$\mathord\downarrow a$ is normal and $p(0)=a$, its inductivity gives
$p(a)\leqslant a$, while monotonicity gives $a=p(0)\leqslant p(a)$.
This proves (3)$\Rightarrow$(2).
\end{proof}

The singleton condition is weaker than requiring inductive rank zero:
larger seeds can still grow. The following result shows that a uniform
one-step bound, when required at every basepoint throughout a variety,
also detects distributivity.

\begin{theorem}\label{DetThmLatticeUniversalRank}
Let $V$ be a variety of lattices, possibly with named constants. Then
\[
r_{\mathcal I}^{\mathrm{univ}}(V)\leqslant1
\quad\Longleftrightarrow\quad
V\text{ is distributive}.
\]
Every non-trivial such distributive variety has
$\operatorname{urk}(V)=(1,1)$.
\end{theorem}

\begin{proof}
First let $L$ be distributive, fix $a\in L$, and let $I\subseteq L$
be non-empty. Put
\[
x\preceq_a y\quad\Longleftrightarrow\quad
x\wedge a\leqslant y\leqslant x\vee a,
\qquad
\theta_I=\operatorname{Cg}_L(I\times\{a\}).
\]
Distributivity shows that $\preceq_a$ is reflexive, transitive and
compatible. We claim that the semicongruence generated by
$I\times\{a\}$ is
\begin{equation}\label{DetEqDistributiveStar}
R_{I,a}=\theta_I\cap\preceq_a.
\end{equation}

For finite $I$, put $\ell=a\wedge\bigwedge I$ and
$h=a\vee\bigvee I$. Both $(\ell,a)$ and $(h,a)$ belong to $R_{I,a}$,
and the polynomial $(s\vee U)\wedge V$ recovers $(s,a)$ from these two
pairs, for each $s\in I$. Thus they generate $R_{I,a}$, and
$\theta_I=\operatorname{Cg}_L(\ell,h)$. Moreover,
\begin{equation}\label{DetEqDistributivePrincipalCongruence}
x\mathrel{\theta_I}y
\quad\Longleftrightarrow\quad
x\wedge\ell=y\wedge\ell\ \text{ and }\ x\vee h=y\vee h.
\end{equation}
To verify this, the right-hand side is the kernel of a lattice
homomorphism and contains $(\ell,h)$. Conversely, in the quotient by
$\operatorname{Cg}_L(\ell,h)$, denote the common image of $\ell,h$ by
$c$. Equality of both meets and both joins with $c$ forces the two
images to be equal, since in a distributive lattice
\[
x=x\wedge(y\vee c)
=(x\wedge y)\vee(x\wedge c)
=(x\wedge y)\vee(y\wedge c)=y.
\]
Now $R_{I,a}\subseteq\theta_I\cap\preceq_a$ follows by compatibility.
For the reverse inclusion, suppose $x\mathrel{\theta_I}y$ and
$x\preceq_a y$. The polynomial
\[
p(U,V)=(x\wedge V)\vee(y\wedge U)\vee(x\wedge y)
\]
satisfies $p(\ell,h)=x$ by
\eqref{DetEqDistributivePrincipalCongruence}, and $p(a,a)=y$ by the
inequalities defining $\preceq_a$. This proves
\eqref{DetEqDistributiveStar} for finite $I$. The general case follows
by taking directed unions over finite non-empty subsets of $I$.

Write $R=R_{I,a}$. Both $RI$ and $IR$ contain $I$ and lie in the
$\theta_I$-class of $a$. Replacing $I$ by either set therefore leaves
$\theta_I$ unchanged, and \eqref{DetEqDistributiveStar} shows that it
also leaves $R$ unchanged. Transitivity now gives
\[
\mathcal I_a^2(I)=R^2I=RI,
\qquad
\mathcal D_a^2(I)=IR^2=IR.
\]

For the converse, let $L\in V$ be non-distributive. Its order relation
\[
B=\{(u,v)\in L^2\mid u\leqslant v\}
\]
is a subalgebra of $L^2$, including any named constants. We use the
following one-step obstruction. If $\alpha=(d,e)\in B$ and
$T\subseteq B$ satisfies $d\leqslant\pi_2(t)$ for every $t\in T$,
then
\begin{equation}\label{DetEqLatticeSquareBarrier}
w\in\mathcal I_\alpha^B(T)
\quad\Longrightarrow\quad
\pi_1(t)\leqslant\pi_2(w)\text{ for some }t\in T.
\end{equation}
Indeed, write $w=q(t_1,\ldots,t_k)$ with
$q(\alpha,\ldots,\alpha)=t\in T$. Compare the first coordinate of
this anchor with the second coordinate of $w$. All variable inputs
increase from $d$ to $\pi_2(t_i)$, and the first coordinate of each
polynomial parameter is at most its second. Monotonicity proves
\eqref{DetEqLatticeSquareBarrier}.

If $L$ contains $M_3$, retain the notation $o,t,a,b,c$ from the
preceding proof and take
\[
\alpha=(a,t),\qquad
T=\{(a,a),(a,t),(b,t)\}.
\]
The first induction contains $(o,t)$, by meeting $(a,t)$ and $(b,t)$.
The polynomial
\[
p(z)=(z\wedge(o,b))\vee(z\wedge(o,c))
\]
has $p(\alpha)=(o,t)$ and $p(a,a)=(o,o)$, so the second induction
contains $(o,o)$. The first does not: neither $a$ nor $b$ is at most
$o$, contradicting \eqref{DetEqLatticeSquareBarrier}.

If $L$ contains $N_5$, with $o<a<b<t$ and remaining element $c$, take
\[
\alpha=(a,t),\qquad
T=\{(b,b),(b,t),(c,t)\}.
\]
The polynomial $q(z)=(z\vee(c,t))\wedge(b,t)$ has
$q(\alpha)=(b,t)\in T$ and $q(c,t)=(o,t)$, so $(o,t)$ belongs to
the first induction. Now
\[
p(z)=(z\wedge(o,a))\vee(z\wedge(o,c))
\]
has $p(\alpha)=(o,t)$ and $p(b,b)=(o,a)$. Thus $(o,a)$ belongs to
the second induction, but not to the first, since neither $b$ nor
$c$ is at most $a$. Again \eqref{DetEqLatticeSquareBarrier} applies.
The forbidden-sublattice criterion therefore proves that
$r_{\mathcal I}^{\mathrm{univ}}(V)\geqslant2$ whenever $V$ is not
distributive. The bounds $o,t$ here are local to the sublattice;
the argument takes place in the full algebra $B$, so no assumption
about named constants in these sublattices is needed.

Finally, in a non-trivial lattice choose $\ell<h$. In its square,
induction at $(\ell,\ell)$ from
$\{(\ell,\ell),(h,h)\}$ adds $(\ell,h)$ by meeting with this
parameter. Hence the universal inductive rank is not zero.
Deduction from a singleton different from the basepoint adds that
basepoint, so the universal deductive rank is not zero either.
\end{proof}

\subsection{Semirings, monoids and cancellation}\label{SubsecDetectMonoids}

The results in this subsection concern the usual semiring and monoid
signatures. Additional basic operations can change the polynomial maps
and therefore the conditions below.

\begin{theorem}\label{ThmDetectRings}
For a semiring $S$ with multiplicative identity $1$, the following
conditions are equivalent:
\begin{enumerate}
\item $S$ is a ring;
\item $0\in\mathcal I_0(\{1\})$;
\item $\mathcal I_0(\{1\})=\mathcal D_0(\{1\})$.
\end{enumerate}
\end{theorem}

\begin{proof}
The ideal generated by $1$ is $S$, so Theorem~\ref{ThmD} gives
$\mathcal I_0(\{1\})=1+S$. If $1+u=0$, then
$x+xu=x(1+u)=0$ for every $x\in S$, and $S$ is a ring.
Conversely, in a ring $1+S=S$, and every $x\in S$ belongs to
$\mathcal D_0(\{1\})$, witnessed by the polynomial
$p(z)=x+(1-x)z$. Finally, deduction from a non-empty set always
contains $0$, so (3) implies (2).
\end{proof}

Next we use the diagonal as a subset of a direct square. It is
important that induction and deduction are taken in the square,
whose parameters need not be diagonal.

\begin{theorem}\label{ThmDetectMonoidDiagonal}
Let $M$ be a monoid, let $e=(1,1)\in M^2$, and put
$\Delta_M=\{(a,a)\mid a\in M\}$. Then
\begin{enumerate}
\item $\Delta_M$ is $e$-inductive if and only if $M$ is commutative;
\item $\Delta_M$ is $e$-deductive if and only if $M$ is commutative
and cancellative.
\end{enumerate}
Moreover, for a commutative monoid,
\begin{equation}\label{EqDetectGroupCompletion}
\mathcal D_e(\Delta_M)
=\{(a,b)\in M^2\mid ac=bc\text{ for some }c\in M\}.
\end{equation}
This relation is the kernel congruence of the canonical homomorphism
from $M$ to its Grothendieck group.
\end{theorem}

\begin{proof}
Suppose first that the diagonal is inductive. For $a,b\in M$, the
polynomial $p(z)=(1,b)z(b,1)$ has
$p(e)=(b,b)$ and $p(a,a)=(ab,ba)$, so $ab=ba$.
If instead the diagonal is deductive, use
$q(z)=(ab,1)z(1,ba)$. Since
$q(a,a)=(aba,aba)$ and $q(e)=(ab,ba)$, commutativity follows again.
If $ac=bc$, apply deductivity to $r(z)=(a,b)z$ at the diagonal
input $(c,c)$ to obtain $a=b$. Thus $M$ is cancellative.

Conversely, in a commutative monoid a polynomial on $M^2$, evaluated
on diagonal inputs, has the form $(au,bu)$, while its value on the
all-$e$ tuple is $(a,b)$. Equality $a=b$ implies $au=bu$, proving
inductivity. Cancellation gives the converse implication, proving
deductivity. The same normal form gives the inclusion from left
to right in~\eqref{EqDetectGroupCompletion}; the reverse inclusion
is witnessed by $r(z)=(a,b)z$ at $(c,c)$.
The identification with the group-completion kernel is the classical
criterion of \cite[Proposition II.1.1(b)]{DetWei13}.
\end{proof}

Thus one deduction step on the diagonal gives the least congruence
with cancellative quotient. It measures precisely the obstruction
to embedding a commutative monoid into its group completion,
whereas induction on the same seed leaves the diagonal unchanged.

We now write commutative monoids additively. Recall that a commutative
monoid is \emph{conical} if $a+b=0$ implies $a=b=0$.

\begin{theorem}\label{ThmDetectPositiveCone}
Let $(M,+,0)$ be a commutative monoid. Then
$M\setminus\{0\}$ is $0$-inductive if and only if $M$ is conical.
Consequently $M$ is isomorphic to the positive cone of a partially
ordered abelian group if and only if
\[
\Delta_M\text{ is $(0,0)$-deductive in }M^2,
\qquad
M\setminus\{0\}\text{ is $0$-inductive in }M.
\]
The empty set is allowed when $M$ is trivial.
\end{theorem}

\begin{proof}
Every polynomial has the form
$c+n_1x_1+\cdots+n_kx_k$, with $n_i\geqslant0$, and its value at
the zero tuple is $c$. In a conical monoid, a sum with $c\ne0$
cannot be zero. Conversely, if $a+b=0$ with $a,b\ne0$, then
$p(z)=a+z$ has nonzero anchor $p(0)=a$ and sends the nonzero input
$b$ to zero, contradicting inductivity.

By Theorem~\ref{ThmDetectMonoidDiagonal}, the two displayed conditions
say exactly that $M$ is cancellative and conical. A cancellative
commutative monoid embeds in its Grothendieck group $G$; identifying
$M$ with its image, define
\[
g\leqslant h\quad\Longleftrightarrow\quad h-g\in M.
\]
This is a translation-invariant preorder, and conicality makes it
antisymmetric. Its positive cone is $M$. Conversely, a positive
cone in an ordered abelian group is cancellative and conical.
This is the usual positive-cone construction; see also
\cite[Proposition 2.5]{DetSti12}.
\end{proof}

For the next test we return to multiplicative notation. A monoid is
\emph{directly finite} if $ab=1$ implies $ba=1$.

\begin{theorem}\label{ThmDetectDirectFinite}
Let $U(M)$ denote the group of units of a monoid $M$. Then $M$ is
directly finite if and only if $M\setminus U(M)$ is $1$-inductive.
\end{theorem}

\begin{proof}
If $ab=1$ and $ba\ne1$, then $a$ and $b$ are nonunits. The polynomial
$p(z)=az$ has nonunit anchor $p(1)=a$ and sends the nonunit $b$ to
$1$, so the set of nonunits is not inductive.
Conversely, in a directly finite monoid, if $ab$ is a unit then
$a$ and $b$ are units: $a$ has a right inverse $b(ab)^{-1}$, which
by direct finiteness is also a left inverse, and then
$b=a^{-1}(ab)$ is a unit. Inductively the same holds for any finite
product. A monoid polynomial is a word in its variables and
parameters. If its value on some inputs is a unit, then every
parameter occurring in the word is a unit, and hence its value at
the all-$1$ tuple is a unit. This proves the required inductivity.
\end{proof}

The classical equivalent condition is that the nonunits form a
two-sided monoid ideal; see \cite[Lemma 2]{DetCR22}. The preceding
theorem describes this condition by induction at the identity,
without imposing commutativity.

\subsection{Universal rank one in monoids and bands}\label{SubsecDetectMonoidRanks}

Here the basepoint ranges over all elements, as in
Subsection~\ref{SubsecUniversalRanks}. The use of universal ranks is
essential: it is different from using only the monoid identity.

\begin{lemma}\label{LemDetectSemilatticeObstructions}
In the variety of meet-semilattices, both universal ranks exceed
one. The same is true when a greatest element is named.
\end{lemma}

\begin{proof}
For induction, take the meet-semilattice of all subsets of
$\{1,2,3\}$, with $a=\{1\}$ and
$S=\{\{1\},\{1,2\},\{3\}\}$. Every nonconstant polynomial is a
meet of a nonempty selection of its variables, possibly with a
parameter. Its anchor at $a$ can belong to $S$ only when it is
$\{1\}$. It follows that
\[
\mathcal I_a(S)=S\cup\{\varnothing\}.
\]
At the next step $p(z)=\{2\}\cap z$ sends the seed element
$\{1,2\}$ to $\{2\}$, using the newly available anchor
$p(a)=\varnothing$. Thus the second step is strictly larger.

For deduction, take the five-element meet-semilattice with order
$x<b<a<1$ and $x<c<1$, where $c$ is incomparable with $a,b$.
At basepoint $a$ the seed $\{c\}$ has first deduction
$\{c,x,a\}$. Indeed, the meet-polynomial normal form shows that
the only new values are $a\wedge d$ for $d\geqslant c$, together
with $a$. The second step contains $b$, witnessed by
$p(z)=b\wedge z$ at the input $x$.
Both examples have a greatest element, and naming it does not
change the polynomial maps.
\end{proof}

\begin{theorem}\label{ThmDetectMonoidRank}
For a variety $V$ in the ordinary monoid signature, the following
conditions are equivalent:
\begin{enumerate}
\item $r_{\mathcal I}^{\mathrm{univ}}(V)\leqslant1$;
\item $r_{\mathcal D}^{\mathrm{univ}}(V)\leqslant1$;
\item $V$ satisfies $x^N=1$ for some integer $N\geqslant1$;
\item $V$ is a Mal'tsev variety.
\end{enumerate}
When these conditions hold and $V$ is non-trivial,
$\operatorname{urk}(V)=(1,1)$.
\end{theorem}

\begin{proof}
Let $E$ be the two-element meet-semilattice with its greatest element
as monoid identity. A monoid identity holds in $E$ precisely when
its two words have the same variable content, after occurrences
of $1$ have been removed. If $E\notin V$, some identity of $V$
has different variable contents. Choose a variable occurring on
just one side and replace all the other variables by $1$; this
gives $x^N=1$ for some $N\geqslant1$.
If (3) fails, therefore, $V$ contains $E$, and hence every
meet-semilattice with a greatest element: these have exactly the
same identities as $E$. Lemma~\ref{LemDetectSemilatticeObstructions}
then rules out both (1) and (2).

If (3) holds with $N>1$, inversion is the term $x^{N-1}$ and
$p(x,y,z)=xy^{N-1}z$ is a Mal'tsev term. For $N=1$ the variety is
trivial. Finally (4) implies (1) and (2) by
Theorem~\ref{ThmMalcevRank}, and the assertion about the exact
pair follows as in Theorem~\ref{PropStandardRanks}.
\end{proof}

Thus rank one detects, among monoid varieties, groups whose
exponents have a common finite bound. This does not assert that
the class of all groups is a variety in the monoid signature.

A \emph{band} is an idempotent semigroup. It is \emph{rectangular}
if it satisfies $xyx=x$.

\begin{theorem}\label{ThmDetectBandRank}
For a variety $V$ in the ordinary band signature, the following
conditions are equivalent:
\begin{enumerate}
\item $r_{\mathcal I}^{\mathrm{univ}}(V)\leqslant1$;
\item $r_{\mathcal D}^{\mathrm{univ}}(V)\leqslant1$;
\item $V$ consists of rectangular bands.
\end{enumerate}
\end{theorem}

\begin{proof}
Suppose that $V$ omits the two-element semilattice. As in the
preceding proof, an identity $u=v$ of $V$ has different variable
contents. Choose a variable $x$ occurring on only one side and
replace every other variable by $y$. The $x$-free side becomes
$y$, and in a band the other side reduces to one of
\[
x,\quad xy,\quad yx,\quad xyx,\quad yxy.
\]
Indeed, repeated adjacent letters disappear, and alternating
words shorten using $(xy)^2=xy$ and $(yx)^2=yx$.
Each resulting identity implies $xyx=x$: the first forces
triviality, the next two give a right-zero or left-zero band,
and the fifth is precisely rectangularity after interchanging the
variables. For the fourth, $xyx=y$ gives $xy=y=yx$ by multiplying
on either side by $x$; interchanging $x,y$ then forces $x=y$. Thus any
nonrectangular band variety contains the semilattices, and
Lemma~\ref{LemDetectSemilatticeObstructions} makes both its
universal ranks greater than one.

Conversely, write a rectangular band in its usual form
$B=L\times R$, with
$(l,r)(l',r')=(l,r')$. For a non-empty subset $S$ and basepoint
$a=(a_L,a_R)$, its generated semicongruence is the product of
\[
\Delta_L\cup\bigl(\pi_L(S)\times\{a_L\}\bigr),
\qquad
\Delta_R\cup\bigl(\pi_R(S)\times\{a_R\}\bigr).
\]
To see this, a non-empty subalgebra of a rectangular band is the
product of its two projections, and multiplication in $B^2$
combines these projections independently. Both displayed
relations are transitive. Induction adds no new coordinate values
to the projections of $S$, whereas deduction can add only the
corresponding coordinates of $a$. Such additions leave the
displayed relations unchanged. Regenerating the semicongruence
after either step therefore makes no change, and transitivity
shows that both operators are idempotent.
\end{proof}

The full variety of rectangular bands has universal pair $(1,1)$.
For induction, in $\{0,1\}\times\{0,1\}$ take basepoint $(0,0)$
and seed $\{(0,1),(1,0)\}$; the element $(1,1)$ is added, for
example by $p(z)=z(0,1)$ at the input $(1,0)$, whose anchor is
$(0,1)$. Deduction is nontrivial whenever the basepoint is omitted
from a non-empty seed. In contrast, left-zero and right-zero
bands have universal pair $(0,1)$, since their polynomial maps
are projections or constants.

\subsection{The number of premises and idempotent addition}
\label{SubsecDetPremises}

The ranks count successive applications of the operators. A different
question concerns the number of elements of the initial set which must
cooperate in a single application. For commutative monoids and
semimodules, a finite uniform bound on this number detects idempotent
addition.

\begin{theorem}\label{ThmDetFinitePremises}
Let $V$ be a variety of commutative monoids, or a subvariety of the
variety of left semimodules over a fixed unital semiring $R$. In either
case use the usual signature and distinguish the additive identity
$0$. The following conditions are equivalent:
\begin{enumerate}
\item $V$ satisfies $x+x=x$.
\item In every $M\in V$, deduction preserves binary unions:
\[
\mathcal D_0(S\cup T)=\mathcal D_0(S)\cup\mathcal D_0(T).
\]
\item There is a positive integer $k$ such that, for every $M\in V$
and every non-empty $S\subseteq M$,
\begin{equation}\label{EqDetFinitePremises}
\mathcal D_0(S)=
\bigcup_{\substack{\varnothing\ne T\subseteq S\\ |T|\leqslant k}}
\mathcal D_0(T).
\end{equation}
\end{enumerate}
When these conditions hold, $k=1$ suffices and
$\mathcal D_0(S)=\mathord\downarrow S$, where
$x\leqslant y$ means $x+y=y$.
\end{theorem}

\begin{proof}
Write $\langle S\rangle$ for the submonoid or subsemimodule generated
by $S$, respectively. Collecting the parameters in a polynomial gives
\begin{equation}\label{EqDetLinearOperators}
\mathcal I_0(S)=S+\langle S\rangle,
\qquad
\mathcal D_0(S)=
\{x\mid x+u\in S\text{ for some }u\in\langle S\rangle\}.
\end{equation}
The monoid case was used in Theorem~\ref{ThmMonoids}; in a semimodule
the same argument uses polynomials of the form
$c+r_1z_1+\cdots+r_nz_n$.

Assume (1). If $x+u=s\in S$, then $x+s=s$, so $x\leqslant s$.
Conversely, if $x\leqslant s\in S$, use $u=s$ in
\eqref{EqDetLinearOperators}. Thus deduction is downward closure,
which proves (2) and (3) with $k=1$. Condition (2) also implies (3)
with $k=1$: apply binary-union preservation to finite sets and then
use finitarity of the polynomial description.

Suppose (3) holds but $a+a\ne a$ in some $M\in V$. In particular,
$a\ne0$. Put $m=k+1$ and work in $M^m$, with coordinates
$0,\ldots,m-1$. Let $e_i$ have entry $a$ in coordinate $i$ and zero
elsewhere, and put
\[
v=(a,\ldots,a),\qquad S=\{v,e_1,\ldots,e_{m-1}\}.
\]
Since $e_0+e_1+\cdots+e_{m-1}=v$, we have
$e_0\in\mathcal D_0(S)$. We show that
$e_0\notin\mathcal D_0(T)$ for every proper subset $T$ of $S$.
Suppose instead that $e_0+u=s\in T$, with
$u\in\langle T\rangle$.
If $v\notin T$, coordinate $0$ gives $a=0$, a contradiction.
Otherwise write
\[
u=\lambda v+\sum_{i=1}^{m-1}\mu_i e_i,
\]
where the coefficients are non-negative integers in the monoid case
and elements of $R$ in the semimodule case. If $e_j\notin T$, its
coefficient can be taken to be zero. Coordinates $0$ and $j$ give
\[
a+\lambda a=s_0,\qquad \lambda a=s_j.
\]
If $s=v$, these equations give $a+a=a$. If $s=e_i\in T$, then
$i\ne j$ and $s_0=s_j=0$, giving $a=0$. Both are contradictions.
Thus all $m=k+1$ elements of $S$ are necessary, contrary to (3).
\end{proof}

\begin{corollary}\label{CorDetIdempotentSemiring}
For a unital semiring $R$, deduction preserves binary unions in every
left $R$-semimodule if and only if addition in $R$ is idempotent.
Equivalently, one-step deduction in the variety of left
$R$-semimodules has a uniform finite bound on the number of seed
elements required.
\end{corollary}

\begin{proof}
If $1+1=1$ in $R$, then $x+x=(1+1)x=x$ in every left
$R$-semimodule. Conversely, the regular left semimodule ${}_RR$
detects failure of idempotent addition in $R$. Apply
Theorem~\ref{ThmDetFinitePremises}.
\end{proof}

The order appearing here is the usual natural order on idempotent
semimodules, including those used in max-plus mathematics; see
\cite[Section~2.2]{DetCGQ04}. The theorem concerns the number of
\emph{distinct seed elements}, including the anchor, and not the
number of variables in a polynomial or the number of successive
deduction steps.

\begin{remark}\label{RemDetInductiveUnions}
For the same classes of varieties, preservation of binary unions by
induction has a different meaning: it holds if and only if the
variety is trivial. Indeed, if $a\ne0$ in $M$, then in $M^2$ the set
$\{(a,0),(0,a)\}$ induces $(a,a)$ by
\eqref{EqDetLinearOperators}. Induction of either singleton alone
cannot do so, since its other coordinate remains zero.
\end{remark}

\subsection{Summary of the detection results}

Table~\ref{TabDetection} collects the principal tests. In variety
rows, conditions are required throughout the variety; universal
operator equalities are required at every basepoint. Unless stated
otherwise, operator equalities hold on all non-empty subsets.
The restrictions on signatures in the preceding subsections apply.

\begingroup
\small
\setlength{\tabcolsep}{4pt}
\setlength{\LTcapwidth}{\textwidth}
\renewcommand{\arraystretch}{1.12}
\makeatletter
\patchcmd{\LT@makecaption}{#1{#2: }}{#1{\textsc{#2.}\enspace}}{}{}
\patchcmd{\LT@makecaption}{#1{#2: }}{#1{\textsc{#2.}\enspace}}{}{}
\makeatother
\begin{longtable}{@{}>{\raggedright\arraybackslash}p{0.49\textwidth}>{\raggedright\arraybackslash}p{0.36\textwidth}l@{}}
\caption{Structural properties detected by induction and deduction.}
\label{TabDetection}\\
\toprule
Setting and test & Property detected & Result\\
\midrule
\endfirsthead
\multicolumn{3}{l}{\textit{Table~\ref{TabDetection}, continued}}\\
\toprule
Setting and test & Property detected & Result\\
\midrule
\endhead
\bottomrule
\endfoot
Variety with $0$: $\mathcal I_0=\mathcal D_0$
 & Subtractivity & \ref{ThmDetSubtractivity}\\
Variety with $0$: $\mathcal I_0^\infty=\mathcal D_0^\infty$
 & Finite $n$-subtractivity; $0$ is isolated in compatible orders
 & \ref{ThmDetEventualSubtractivity}\\
Arbitrary variety: $\mathcal I_a=\mathcal D_a$ universally
 & Mal'tsev property & \ref{ThmDetUniversalAgreement}\\
Arbitrary variety: $\mathcal I_a^\infty=\mathcal D_a^\infty$ universally
 & Finite congruence permutability; no nontrivial compatible orders
 & \ref{ThmDetUniversalAgreement}\\
Bounded lattice: every singleton is $0$-inductive, or its deduction is normal
 & Distributivity & \ref{DetThmLatticeSingletons}\\
Lattice variety: $r_{\mathcal I}^{\mathrm{univ}}\leqslant1$
 & Distributivity & \ref{DetThmLatticeUniversalRank}\\
Unital semiring: $0\in\mathcal I_0(\{1\})$
 & Additive inverses & \ref{ThmDetectRings}\\
Monoid: the diagonal in its square is inductive at $(1,1)$
 & Commutativity & \ref{ThmDetectMonoidDiagonal}\\
Monoid: the diagonal in its square is deductive at $(1,1)$
 & Commutativity and cancellation & \ref{ThmDetectMonoidDiagonal}\\
Commutative monoid: nonzero elements are $0$-inductive
 & Conicality; with cancellation, an ordered-group positive cone
 & \ref{ThmDetectPositiveCone}\\
Monoid: nonunits are $1$-inductive
 & Direct finiteness & \ref{ThmDetectDirectFinite}\\
Monoid variety: either universal rank is at most one
 & Groups of uniformly bounded exponent & \ref{ThmDetectMonoidRank}\\
Band variety: either universal rank is at most one
 & Rectangularity & \ref{ThmDetectBandRank}\\
Commutative monoid or semimodule variety: deduction has a uniform finite premise bound
 & Idempotent addition; the bound is one & \ref{ThmDetFinitePremises}\\
\end{longtable}
\endgroup

\end{document}